\documentclass{amsart}

\usepackage{amsmath,amsfonts,amssymb,amsthm,graphicx,cite,mathrsfs}

\newtheorem{theorem}{Theorem}[section]
\newtheorem{lemma}{Lemma}[section]
\newtheorem{proposition}{Proposition}[section]
\newtheorem{corollary}{Corollary}[section]

\newtheorem{definition}{Definition}[section]

\theoremstyle{remark}
\newtheorem{remark}{Remark}[section]

\begin{document}

\title[Multivariable Bessel polynomials]{Multivariable Bessel polynomials related to the hyperbolic Sutherland model with external Morse potential}

\author{Martin Halln\"as} \address{SISSA, Via Beirut 2-4, 34014 Trieste TS, Italy}
\email{hallnas@sissa.it}

\date{\today}

\begin{abstract}
A multivariable generalisation of the Bessel polynomials is introduced and studied. In particular, we deduce their series expansion in Jack polynomials, a limit transition from multivariable Jacobi polynomials, a sequence of algebraically independent eigenoperators, Pieri type recurrence relations, and certain orthogonality properties. We also show that these multivariable Bessel polynomials provide a (finite) set of eigenfunctions of the hyperbolic Sutherland model with external Morse potential.
\end{abstract}

\maketitle

\section{Introduction}
The Bessel polynomials first appeared in a paper by Bochner \cite{Bochner} as the polynomial solutions of the differential equation
\begin{equation*}
	x^2\frac{d^2y}{dx^2} + (ax + b)\frac{dy}{dx} = n(n + a - 1)y,
\end{equation*}
where $a$ and $b$ are (real) parameters, and $n$ a non-negative integer. The first systematic and detailed study of their properties was later undertaken by Krall and Frink \cite{KF49}. For further details and references see e.g.\ the book by Grosswald \cite{Gross}. In the present paper we introduce and study a multivariable generalisation of the Bessel polynomials as eigenfunctions of the partial differential operator
\begin{equation}
\label{BesselEigenOp}
	D^B = \sum_{i=1}^n x_i^2\frac{\partial^2}{\partial x_i^2} + \sum_{i=1}^n(ax_i + b)\frac{\partial}{\partial x_i} + 2\kappa\sum_{i\neq j}\frac{x_i^2}{x_i-x_j}\frac{\partial}{\partial x_i}
\end{equation}
with $\kappa$ a (real) parameter. In particular, we obtain their series expansion in Jack polynomials, a limit transition from Jacobi polynomials associated with the root system $BC_n$, $n$ algebraically independent eigenoperators, Pieri type recurrence relations, and, as further discussed below, certain orthogonality properties. We note that by setting $n = 1$ in \eqref{BesselEigenOp} we indeed recover the eigenoperator of the one-variable Bessel polynomials. As in the one-variable case, the dependence on the parameter $b$ can be removed by substituting $(bx_1,\ldots,bx_n)$ for $(x_1,\ldots,x_n)$. Following the standard convention we will accordingly set $b = 2$.

The differential operator \eqref{BesselEigenOp} is closely related to the Schr\"odinger operator of the so-called hyperbolic Sutherland model with external Morse potential. The results we obtain on the multivariable Bessel polynomials are thus relevant also for this quantum many-body model. In order to make this relation precise we let
\begin{equation*}
	\Psi_0(x_1,\ldots,x_n) = \prod_{i=1}^n x_i^{(a-2)/2}e^{-\frac{1}{x_i}}\prod_{i<j}(x_i - x_j)^\kappa.
\end{equation*}
As a direct computation shows, conjugation of $D^B$ by this function $\Psi_0$, and a change of coordinates to $(z_1,\ldots,z_n) = (e^{x_1},\ldots,e^{x_n})$, yields the Schr\"odinger type operator
\begin{equation}\label{SchrodOp}
\begin{split}
	H &:= -\Psi_0(D^B - E_0)\Psi_0^{-1}\\ &= -\sum_{i=1}^n\frac{\partial^2}{\partial z_i^2} + \sum_{i=1}^n\left(e^{-2z_i} - (2 - a)e^{-z_i}\right)\\ &\quad + 2\kappa(\kappa - 1)\sum_{i<j}\frac{1}{4\sinh^2\frac{1}{2}(z_i-z_j)}
\end{split}
\end{equation}
with
\begin{equation*}
	E_0 = -\frac{\kappa^2}{3}n(n^2 - 1) - \frac{\kappa(a-(1+\kappa))}{2}n(n - 1) - \frac{(a-1)^2}{4}n;
\end{equation*}
see e.g.~Halln\"as and Langman \cite{HL07} for a proof of this fact. The Schr\"odinger type operator \eqref{SchrodOp} was studied by Inozemtsev and Meshcheryakov \cite{IM86}. They determined the discrete part of its spectrum, and showed that the corresponding eigenfunctions are given by symmetric polynomials. We also mention that the (hyperbolic) Sutherland model, obtained from \eqref{SchrodOp} by eliminating the external Morse potential, was introduced and studied by Sutherland \cite{Sut71}.

It is clear from \eqref{SchrodOp} that any eigenfunction of $D^B$ yields an eigenfunction of $H$ upon multiplication by the function $\Psi_0$. The physically relevant eigenfunctions are those contained in the Hilbert space $L^2(\mathbb{R}^n,dz_1,\ldots dz_n)$. From this point of view the eigenfunctions of $D^B$ contained in $L^2(\mathbb{R}^n_+,|\Psi_0|^2(x)dx_1,\ldots dx_n)$ are of particular interest. In Section 7 we will show that all multivariable Bessel polynomials up to a given degree are contained in this latter Hilbert space if and only if the parameters $a$ and $\kappa$ satisfy a simple inequality. We then prove that these multivariable Bessel polynomials form an orthogonal system with respect to the Hilbert space inner product. In addition, we deduce an explicit expression for the corresponding squared norms. In effect, this provides a (finite) set of normalised eigenfunctions of $H$, representing bound states in the corresponding quantum many-body system.

Multivariable generalisations of the classical orthogonal (Hermite, Laguerre and Jacobi) polynomials, of the same type as the multivariable Bessel polynomials considered here, have been extensively studied in the literature. We mention, in particular, the closely related work of Lassalle \cite{Las91a,Las91b,Las91c} and Macdonald \cite{Mac}, as well as Baker and Forrester \cite{BF97} and van Diejen \cite{vD97,vD99}. For further references see e.g.~the book by Dunkl and Xu \cite{DX01}. As is well known in the one-variable case, the Bessel polynomials share a number of properties with these classical orthogonal polynomials, e.g., the fact that they solve a second order ordinary differential equation. The present study of multivariable Bessel polynomials thus provide a natural complement to the extensive literature alluded to above. We also mention that multivariable generalisations of the Bessel polynomials, different from the one considered here, have been introduced and studied by Exton \cite{Exton}, as well as by Pathan and Bin-Saad \cite{PatBin}.

We conclude this introduction by a brief overview of the paper. In Section 2 we briefly review some basic facts on Jack- and multivariable Jacobi polynomials, as well as generalised hypergeometric series, which will be used in the paper. In Section 3 we give a precise definition of our multivariable Bessel polynomials and deduce their expansion in Jack polynomials. In Section 4 we establish a limit transition from multivariable Jacobi polynomials associated with the root system $BC_n$. This limit transition is then used in Section 5 to obtain $n$ algebraically independent eigenoperators of these multivariable Bessel polynomials. In Section 6 we continue to apply this limit transition to obtain Pieri type recurrence relations. As already mentioned above, the orthogonality of the multivariable Bessel polynomials is discussed in Section 7. In Section 8 we present some further remarks on the relation between the multivariable Bessel polynomials and the eigenfunctions of the hyperbolic Sutherland model with external Morse potential. In this section we also discuss the problem of constructing a (moment) functional with respect to which all multivariable Bessel polynomials would be orthogonal.

\section{Preliminaries}
In this section we recall some basic facts on Jack polynomials and multivariable Jacobi polynomials associated with the root system $BC_n$, as well as on generalised hypergeometric series, which we will make use of in later parts of the paper.

\subsection{Symmetric functions and Jack polynomials}
We recall that a partition $\lambda = (\lambda_1,\lambda_2,\ldots)$ is a finite sequence of non-negative integers such that
\begin{equation*}
	\lambda_1\geq\lambda_2\geq\cdots.
\end{equation*}
The number of non-zero parts is commonly referred to as the length of $\lambda$, and denoted by $\ell(\lambda)$. For simplicity of exposition we will not distinguish two partitions that differ only by a string of zeros at the end. The weight of a partition $\lambda$ is the sum
\begin{equation*}
	|\lambda| = \lambda_1 + \lambda_2 +\cdots
\end{equation*}
of its parts. The diagram of $\lambda$ is the set of points $(i,j)\in\mathbb{Z}^2$ such that $1\leq j\leq\lambda_i$. By reflection in the main diagonal the conjugate partition $\lambda^\prime$ is obtained. We let $e_i$ be the unit vector in $\mathbb{Z}^n$ defined by $(e_i)_j = \delta_{ij}$, and write
\begin{equation*}
	\lambda^{(i)} = \lambda + e_i,\quad \lambda_{(i)} = \lambda - e_i
\end{equation*}
for $i = 1,\ldots,n$. On the set of partitions the dominance order can be defined by
\begin{equation*}
	\mu\leq\lambda \Leftrightarrow |\mu| = |\lambda|~\text{and}~\mu_1 +\cdots +\mu_i\leq \lambda_1 +\cdots + \lambda_i,\quad \forall i = 1,\ldots,n.
\end{equation*}
We stress that two partitions $\lambda$ and $\mu$ are comparable only if they have the same weight, i.e., if $|\mu| = |\lambda|$. The monomial symmetric polynomials $m_\lambda$ are given by
\begin{equation*}
	m_\lambda(x_1,\ldots,x_n) = \sum_{\alpha\in S_n\lambda} x_1^{\alpha_1}\cdots x_n^{\alpha_n},
\end{equation*}
where $S_n\lambda$ denotes the orbit of $\lambda$ under the natural action of the permutation group $S_n$ of $n$ objects. These monomials span the space
\begin{equation*}
	\Lambda_n\subset\mathbb{C}\lbrack x_1,\ldots,x_n\rbrack
\end{equation*}
of symmetric polynomials in the variables $x = (x_1,\ldots,x_n)$ with complex coefficients. For any non-negative integer $m$ we let
\begin{equation*}
	\Lambda_n^{\leq m} = \bigoplus_{k = 0}^m \Lambda_n^k,
\end{equation*}
where $\Lambda_n^k\subset\Lambda_n$ is the subspace of homogeneous polynomials of degree $k$. For further details on partitions and symmetric functions see e.g.~Chapter I in Macdonald \cite{Mac95}.

In order to proceed we require the differential operators
\begin{equation*}
	E_\ell = \sum_{i=1}^n x_i^\ell \frac{\partial}{\partial x_i}
\end{equation*}
and
\begin{equation*}
	D_k = \sum_{i=1}^n x_i^k \frac{\partial^2}{\partial x_i^2} + 2\kappa\sum_{i\neq j}\frac{x_i^k}{x_i-x_j}\frac{\partial}{\partial x_i}
\end{equation*}
for $\ell = 0,1$ and $k = 1,2$, respectively. We recall that the (monic) Jack polynomials $P_\lambda$ can be defined as the unique symmetric polynomials that satisfy the following two conditions:
\begin{enumerate}
\item $P_\lambda = m_\lambda + \sum_{\mu<\lambda}u_{\lambda\mu}m_\mu$ for some coefficients $u_{\lambda\mu}$,
\item $P_\lambda$ is an eigenfunction of $E_1$ and $D_2$;
\end{enumerate}
see e.g.~Macdonald \cite{Mac95} or Stanley \cite{Sta89}. To any two partitions $\lambda$ and $\mu$ a generalised binomial coefficient $\binom{\lambda}{\mu}$ can be associated via the expansion
\begin{equation*}
	\frac{P_\lambda(x_1+1,\ldots,x_n+1)}{P_\lambda(1^n)} = \sum_{\mu\subseteq\lambda}\binom{\lambda}{\mu}\frac{P_\lambda(x_1,\ldots,x_n)}{P_\lambda(1^n)}.
\end{equation*}
We recall that the fact that the sum extends only over partitions $\mu\subseteq\lambda$ is non-trivial. To  the best of my knowledge, this fact was first established by Kaneko \cite{Kan93}. A short and elegant proof, exploiting a relation with the so-called shifted Jack polynomials, was later given by Okounkov and Olshanski \cite{OO97}.

For future reference we state the known action of the operators $E_\ell$ and $D_k$, as well as the power sum $p_1(x) = x_1 +\cdots + x_n$, on the Jack polynomials:
\begin{subequations}\label{actionFormulae}
\begin{align}
\label{E0Action}
	E_0 \frac{P_\lambda(x)}{P_\lambda(1^n)} &= \sum_i \binom{\lambda}{\lambda_{(i)}}\frac{P_{\lambda_{(i)}}(x)}{P_{\lambda_{(i)}}(1^n)},\\
\label{E1Action}
	E_1 P_\lambda &= |\lambda|P_\lambda,\\
\label{D1Action}
	D_1 \frac{P_\lambda(x)}{P_\lambda(1^n)} &= \sum_i \binom{\lambda}{\lambda_{(i)}}(\lambda_i - 1 + \kappa(n - i))\frac{P_{\lambda_{(i)}}(x)}{P_{\lambda_{(i)}}(1^n)},\\
\label{D2Action}
	D_2 P_\lambda &= d_\lambda P_\lambda,\quad d_\lambda = \sum_i \lambda_i(\lambda_i - 1 + 2\kappa(n-i)),\\
\label{p1Action}
	p_1(x) \frac{P_\lambda(x)}{h_\lambda} &= \sum_i \binom{\lambda^{(i)}}{\lambda}\frac{P_{\lambda^{(i)}}(x)}{h_{\lambda^{(i)}}},
\end{align}
\end{subequations}
where we have made use of the notation
\begin{equation*}
	h_\lambda = \prod_{i=1}^{\ell(\lambda)}\prod_{j=1}^{\lambda_i}(\lambda_i - j + \kappa(\lambda^\prime_j - i) + 1),
\end{equation*}
and where the sums are over all $i = 1,\ldots,n$ such that in \eqref{E0Action} and \eqref{D1Action} $\lambda_{(i)}$ is a partition and in \eqref{p1Action} $\lambda^{(i)}$ is a partition. We remark that \eqref{E1Action} and \eqref{D2Action} are readily inferred from the definition of the Jack polynomials as common eigenfunctions of $E_1$ and $D_2$, and the action of these operators on the monomials $m_\lambda$. Equation \eqref{E0Action} can be deduced from the definition of the binomial coefficients (see e.g.\ Kaneko \cite{Kan93}), and \eqref{D1Action} can then be obtained using the fact that $2D_1 = \lbrack E_0, D_2\rbrack$. Equation \eqref{p1Action}, on the other hand, is more difficult to prove. As far as I know, it was first established by Kaneko \cite{Kan93}; see also Okounkov and Olshanski \cite{OO97}.

\subsection{Generalised hypergeometric series}
Kaneko \cite{Kan93}, Kor\'anyi \cite{Kor91} and Macdonald \cite{Mac} introduced and studied multivariable hypergeometric series associated with Jack polynomials. In order to recall their definition we require the following generalisation of the Pochhammer symbol:
\begin{equation*}
	\lbrack\alpha\rbrack^{(\kappa)}_\lambda = \prod_{i=1}^{\ell(\lambda)}\lbrack\alpha - \kappa(i - 1)\rbrack_{\lambda_i},	
\end{equation*}
where the ordinary Pochhammer symbol $\lbrack\alpha\rbrack_m$ is given by $\lbrack\alpha\rbrack_0 = 1$ and $\lbrack\alpha\rbrack_m = \alpha(\alpha+1)\cdots(\alpha+m-1)$ for $m>0$. The generalised hypergeometric series can now be defined by its series expansion in Jack polynomials:
\begin{equation}
\label{HypergSeries}
	{}_{p}F_q(\alpha_1,\ldots,\alpha_p;\beta_1,\ldots,\beta_q;x;\kappa) = \sum_\lambda \frac{\lbrack\alpha_1\rbrack^{(\kappa)}_\lambda \cdots \lbrack\alpha_p\rbrack^{(\kappa)}_\lambda}{\lbrack\beta_1\rbrack^{(\kappa)}_\lambda \cdots \lbrack\beta_q\rbrack^{(\kappa)}_\lambda}\frac{P_\lambda(x;\kappa)}{h_\lambda(\kappa)}.
\end{equation}
For $p = 2$ and $q=1$ it is known that the resulting generalised hypergeometric series ${}_2F_1(\alpha_1,\alpha_2;\beta_1;x;\kappa)$ satisfies the diferential equation
\begin{equation}
\label{HypergDiffEq}
	(D_1 - D_2)F + (\beta_1 - \kappa(n-1))E_0 F - (\alpha_1 + \alpha_2 + 1 - \kappa(n-1))E_1 F = n\alpha_1\alpha_2 F;
\end{equation}
see e.g.~Kaneko \cite{Kan93} or Yan \cite{Yan92}.

\subsection{Jacobi polynomials associated with the root system $BC_n$}
Jacobi polynomials associated with root systems have been extensively studied in the literature. We mention, in particular, the work of Vretare \cite{Vre84}, the first to introduce such polynomials, and that of Debiard \cite{Deb88}, Heckman and Opdam \cite{HO87}, as well as that of Beerends and Opdam \cite{BO93}. We will make use of the Jacobi polynomials associated with the root systems $BC_n$, henceforth referred to as the $BC_n$ Jacobi polynomials. These polynomials are even and permutation invariant trigonometric polynomials, and can be defined as eigenfunctions of the differential operator
\begin{multline*}
	D^{BC} = \sum_{i=1}^n\frac{\partial^2}{\partial z_i^2} + \sum_{i=1}^n\left(k_1\coth\frac{1}{2}z_i + 2k_2\coth z_i\right)\frac{\partial}{\partial z_i}\\ + k_3\sum_{i<j}\left(\coth\frac{1}{2}(z_i - z_j)\left(\frac{\partial}{\partial z_i} - \frac{\partial}{\partial z_j}\right) + \coth\frac{1}{2}(z_i + z_j)\left(\frac{\partial}{\partial z_i} + \frac{\partial}{\partial z_j}\right)\right)
\end{multline*}
with $k_1,k_2$ and $k_3$ (real) parameters. As discussed by Beerends and Opdam \cite{BO93}, by a change of variables  to $t_i = -\sinh^2 z_i/2$ this operator can be written as
\begin{equation*}
\begin{split}
	D^{BC}  &= \sum_{i=1}^n t_i(t_i - 1)\frac{\partial^2}{\partial t_i^2} - \sum_{i=1}^n\left(k_1+k_2+\frac{1}{2}-(k_1+2k_2+1)t_i\right)\frac{\partial}{\partial t_i}\\ &\quad + 2k_3\sum_{i\neq j}\frac{t_i(t_i - 1)}{t_i-t_j}\frac{\partial}{\partial t_i}.
\end{split}
\end{equation*}
It is thus clear from formulae \eqref{actionFormulae} that
\begin{equation*}
	D^{BC}P_\lambda(t(z)) = e^{BC}_\lambda P_\lambda(t(z)) + \sum_{\mu\subset\lambda}c_{\lambda\mu}P_\mu(t(z))
\end{equation*}
for some coefficients $c_{\lambda\mu}$, and with
\begin{equation*}
	e^{BC}_\lambda = d_\lambda + (k_1 + 2k_2 + 1)|\lambda|.
\end{equation*}
This leads us to the definition of the $BC_n$ Jacobi polynomials $P^{BC}_\lambda$ as the unique polynomials satisfying the following two conditions:
\begin{enumerate}
\item $P^{BC}_\lambda(z) =  (-4)^{|\lambda|}P_\lambda(t(z)) + \sum_{\mu\subset\lambda}u_{\lambda\mu}P_\lambda(t(z))$ for some coefficients $u_{\lambda\mu}$,
\item $P^{BC}_\lambda$ is an eigenfunction of $D^{BC}$.
\end{enumerate}
We remark that it is common to define the $BC_n$ Jacobi polynomials in terms of the monomials $m_\lambda(e^{z_1}+e^{-z_1},\ldots,e^{z_n}+e^{-z_n})$; see e.g. Beerends and Opdam \cite{BO93}. However, using the triangular structure in the expansion of the Jack polynomials in the symmetric monomials $m_\lambda$ it is easy to see that this leads to an equivalent definition. As will become apparent below, for our purposes the present definition is somewhat more convenient. We also mention that the leading coefficient of the $BC_n$ Jacobi polynomials $P^{BC}_\lambda$ has been set to $(-4)^{|\lambda|}$ in order to ensure that they are monic when expanded in the monomials specified above. 

\section{Definition and series expansion in Jack polynomials}
In this section we will first give a precise definition of our multivariable Bessel polynomials, and then proceed to deduce and study their expansion in Jack polynomials. We start by observing that \eqref{actionFormulae} implies that the operator $D^B$ acts on the Jack polynomials as follows:
\begin{equation}
\label{DBAction}
	D^B P_\lambda = e_\lambda P_\lambda + 2\sum_{i=1}^n\frac{P_\lambda(1^n)}{P_{\lambda_{(i)}}(1^n)}\binom{\lambda}{\lambda_{(i)}}P_{\lambda_{(i)}}(x)
\end{equation}
with
\begin{equation}
\label{BesselEigenvalues}
	e_\lambda = d_\lambda + a|\lambda|
\end{equation}
(where we recall that $b = 2$). This implies, in particular, that $D^B$ maps each Jack polynomial $P_\lambda$ to a linear combination of Jack polynomials $P_\mu$ labelled by partitions $\mu\subseteq\lambda$. It follows that $D^B$ have eigenfunctions of the form
\begin{equation*}
	Y_\lambda(x;\kappa) = \sum_{\mu\subseteq\lambda}u_{\lambda\mu}P_\mu(x;\kappa)
\end{equation*}
for some coefficients $u_{\lambda\mu}$. Indeed, using \eqref{DBAction} it is readily verified that this is the case if and only if the coefficients $u_{\lambda\mu}$ satisify the recurrence relation
\begin{equation}
\label{BesselRecursRel}
	(e_\lambda - e_\mu)u_{\lambda\mu} = 2\sum_{i=1}^n\frac{P_{\mu^{(i)}}(1^n)}{P_\mu(1^n)}\binom{\mu^{(i)}}{\mu}u_{\lambda\mu^{(i)}}
\end{equation}
for all partitions $\mu\subset\lambda$. If $e_\mu\neq e_\lambda$ for all such partitions, then the coefficients $u_{\lambda\mu}$ with $\mu\subset\lambda$ are uniquely determined by \eqref{BesselRecursRel} once the value of $u_{\lambda\lambda}$ is fixed. This leads us to our definition of multivariable Bessel polynomials.

\begin{definition}\label{besselDef}
Fix the values of the parameters $a$ and $\kappa$, and let $\lambda = (\lambda_1,\ldots,\lambda_n)$ be a partition such that $e_\mu\neq e_\lambda$ for all partitions $\mu\subset\lambda$. We then define $Y_\lambda$ as the (unique) symmetric polynomial such that
\begin{enumerate}
\item $Y_\lambda = P_\lambda + \sum_{\mu\subset\lambda}u_{\lambda\mu}P_\mu$ for some coefficients $u_{\lambda\mu}$,
\item $Y_\lambda$ is an eigenfunction of $D^B$.
\end{enumerate}
\end{definition}

It is clear from the discussion above that the level of non-degeneracy required of the eigenvalues $e_\lambda$ is essential in order for the multivariable Bessel polynomials $Y_\lambda$ to be well-defined. As we now show, this condition is generically satisfied.

\begin{proposition}\label{nonDegeneracyProp}
Let $\lambda = (\lambda_1,\ldots,\lambda_n)$ be a partition. Then, for generic values of the parameters $a$ and $\kappa$, $e_\mu\neq e_\lambda$ for all partitions $\mu\subset\lambda$. In particular, this is the case if $\kappa\geq 0$ and $a$ satisfies the condition
\begin{equation}\label{aInequality}
	a < -2(|\lambda| + \kappa(n-1)) + 1.
\end{equation}
\end{proposition}

\begin{proof}
We fix a partition $\mu\subset\lambda$. Since $e_\lambda - e_\mu$ is a polynomial in $\kappa$ and $a$, and not identically zero, we have that $e_\mu\neq e_\lambda$ for $(\kappa,a)$ in an open dense subset of $\mathbb{R}^2$ (or even $\mathbb{C}^2$). We observe that $\lambda$ is of the form
\begin{equation*}
	\lambda = \mu + \sum_{i=1}^n n_ie_i,\quad n_i\in\mathbb{N}.
\end{equation*}
It follows that
\begin{equation*}
	e_\lambda - e_\mu = \sum_{i=1}^n n_i\big(2(\mu_i + \kappa(n-i)) - 1 + a\big).
\end{equation*}
Under the stated conditions on $\kappa$ and $a$ the right hand side is clearly non-zero.
\end{proof}

Throughout the remainder of the paper we will assume that the parameters $a$ and $\kappa$ are such that the Bessel polynomials in question are well-defined. We mention, already at this point, that in our discussion of orthogonality in Section 7 the relevant parameter values will be precisely those for which $\kappa\geq 0$ and $a$ satisfies the condition \eqref{aInequality}. Proposition \ref{nonDegeneracyProp} will then ensure that the multivariable Bessel polynomials under consideration indeed are well-defined. With that said we proceed to solve the recurrence relation \eqref{BesselRecursRel}, and thereby to obtain an expression for the coefficients in the series expansion of the multivariable Bessel polynomial in Jack polynomials. By iterating \eqref{BesselRecursRel} we find that the coefficients $u_{\lambda\mu}$ are given by
\begin{equation*}
	u_{\lambda\mu} = 2^{|\lambda|-|\mu|}\sum_T \prod_{i=1}^{|\lambda|-|\mu|}\frac{P_{{}^{(i-1)}\lambda}(1^n)}{P_{{}^{(i)}\lambda}(1^n)}\binom{{}^{(i-1)}\lambda}{{}^{(i)}\lambda}\Big/(e_\lambda - e_{{}^{(i)}\lambda}),
\end{equation*}
where the sum is over all chains of partitions
\begin{equation*}
	T: \lambda = {}^{(0)}\lambda\supset {}^{(1)}\lambda\supset\ldots \supset {}^{(r)}\lambda = \mu
\end{equation*}
such that $r = |\lambda| - |\mu|$ and $|{}^{(i-1)}\lambda| - |{}^{(i)}\lambda| = 1$ for all $i = 1,\ldots,r$, i.e., over all standard tableux $T$ of shape $\lambda/\mu$; see e.g.~Section I.1 in Macdonald \cite{Mac95}. Inserting the expression \eqref{BesselEigenvalues} for the eigenvalues of the Bessel polynomials we thus arrive at the following:

\begin{proposition}\label{seriesProp}
The multivariable Bessel polynomials have a series expansion
\begin{equation*}
	Y_\lambda = P_\lambda + \sum_{\mu\subset\lambda} u_{\lambda\mu}P_\mu
\end{equation*}
with coefficients
\begin{equation}\label{coeffs}
	u_{\lambda\mu} = 2^{|\lambda|-|\mu|}\sum_T \prod_{i=1}^{|\lambda|-|\mu|}\frac{P_{{}^{(i-1)}\lambda}(1^n)}{P_{{}^{(i)}\lambda}(1^n)}\binom{{}^{(i-1)}\lambda}{{}^{(i)}\lambda}\Big/\left(ai + d_\lambda - d_{{}^{(i)}\lambda}\right),
\end{equation}
where the latter sum is over all standard tableux $T$ of shape $\lambda/\mu$.
\end{proposition}

\begin{remark}
For the multivariable classical orthogonal polynomials this type of series expansions are well known. To the best of my knowledge, such a series expansion was first obtained by Constantine \cite{Con66} for the multivariable Laguerre polynomials corresponding to $\kappa = 1/2$. We recall that for this particular parameter value the Jack polynomials specialise to the so-called Zonal polynomials; see e.g.~Macdonald \cite{Mac95}. James and Constantine \cite{JC74} later obtained the analogous series expansion for the multivariable Jacobi polynomials with $\kappa = 1/2$ or $1$. Series expansions for general values of $\kappa$ can be found in Baker and Forrester \cite{BF97}, Beerends and Opdam \cite{BO93}, Lassalle \cite{Las91a,Las91b,Las91c}, as well as in Macdonald \cite{Mac}.\end{remark}

The generalised binomial coefficients $\binom{\lambda}{\mu}$ are in general rather complicated objects; see e.g.the remark in Section 3 of Okounkov and Olshanski \cite{OO97}. However, the particular binomial coefficients appearing in Proposition \ref{seriesProp} can be written down explicitly using the fact that
\begin{equation}\label{binomExp}
	\binom{\lambda^{(i)}}{\lambda} = (\lambda_i + 1 + \kappa(\ell(\lambda) - i))\prod_{j\neq i}\frac{\lambda_i-\lambda_j+1+\kappa(j-i-1)}{\lambda_i-\lambda_j+1+\kappa(j-i)}.
\end{equation}
This latter formula can be deduced from a comparison of \eqref{p1Action} with the corresponding Pieri formula in Macdonald's book \cite{Mac95} (Equation 6.24 (iv) in Chapter VI); see Lassalle \cite{Las98} for further details. We recall that Stanley \cite{Sta89} (see Theorem 5.4) obtained the specialisation of  the integral form
\begin{equation*}
	J_\lambda = \kappa^{-|\lambda|}h^\lambda P_\lambda,\quad h^{\lambda} = \prod_{i=1}^{\ell(\lambda)}\prod_{j=1}^{\lambda_i}(\lambda_i - j + \kappa(\lambda^\prime_j - i + 1)),
\end{equation*}
of the Jack polynomials at $x = (1^n)$. Combining this specialisation formula with the appropriate limit of Proposition 3 in Lassalle \cite{Las98} (multiply by $1 - t$, set $q=t^\alpha$ and let $t\rightarrow 1$) it is readily verified that
\begin{equation}\label{JackExp}
\begin{split}
	\frac{P_{\lambda^{(i)}}(1^n)}{P_\lambda(1^n)} &= \kappa\frac{\lambda_i+\kappa(n-i+1)}{\lambda_i+\kappa(\ell(\lambda)-i+1)}\prod_{j<i}\frac{\lambda_i-\lambda_j+1+\kappa(j-i)}{\lambda_i-\lambda_j+1+\kappa(j-i-1)}\\ &\quad\times\prod_{j>i}\frac{\lambda_i-\lambda_j+\kappa(j-i+1)}{\lambda_i-\lambda_j+\kappa(j-i)}.
\end{split}
\end{equation}
Using these two formulae we can thus write down an explicit expression for the coefficients $u_{\lambda\mu}$. This enables us to study the dependence of the multivariable Bessel polynomials on the parameters $a$ and $\kappa$. In particular, since the coefficients in the expansion of the integral form
of the Jack polynomials in terms of the monomials $m_\lambda$ are polynomials in $1/\kappa$ (see e.g.~Knop and Sahi \cite{KS97} and references therein), we can prove the following:

\begin{corollary}\label{analyticCorollary}
The multivariable Bessel polynomials are rational functions of $a$ and $\kappa$, and are thus for generic values analytic functions of these parameters. In particular, this is the case if $\kappa\geq 0$ and $a$ satisfies the inequality \eqref{aInequality}.
\end{corollary}

\begin{proof}
It is immediate from the discussion above that the coefficients $u_{\lambda\mu}$ are rational functions of $a$ and $\kappa$. The first part of the statement can thus be established via an argument similar to that used to prove Proposition \ref{nonDegeneracyProp}. The second part of the statement follows from Proposition \ref{nonDegeneracyProp}, inspection of \eqref{binomExp} and \eqref{JackExp}, and the fact that the Jack polynomials $P_\lambda$ are analytic in $\kappa$ for $\kappa\geq 0$.
\end{proof}

In general, it seems hard to deduce a closed expression for the coefficients $u_{\lambda\mu}$. However, if $\lambda = (k^n)$ for some non-negative integer $k$ the situation simplifies considerably, as the multivariable Bessel polynomials then are given by a terminating generalised hypergeometric series. More precisely, we have the following:

\begin{proposition}
For any non-negative integer $k$,
\begin{equation*}
	Y_{(k^n)}(x;a,\kappa) = C_k(a,\kappa) {}_2F_0(-k,k +a -1 -\kappa(n-1);-x/2;\kappa),
\end{equation*}
where
\begin{equation*}
	C_k(a,\kappa) = \frac{h_{(k^n)}(\kappa)}{\lbrack-k\rbrack^{(\kappa)}_{(k^n)}\lbrack k+a-1-\kappa(n-1)\rbrack^{(\kappa)}_{(k^n)}}.
\end{equation*}
\end{proposition}

\begin{proof}
Since $(-k)_\lambda = 0$ unless $\lambda\subseteq (k^n)$, we obtain upon setting $p = 2$, $q = 1$ and $\alpha_1 = -k$ in \eqref{HypergSeries} that
\begin{equation*}
	{}_2F_1(-k,\alpha_2;\beta_1;x) = \sum_{\lambda\subseteq (k^n)}\frac{\lbrack-k\rbrack^{(\kappa)}_\lambda \lbrack\alpha_2\rbrack^{(\kappa)}_\lambda}{\lbrack\beta_1\rbrack^{(\kappa)}_\lambda}\frac{P_\lambda(x)}{h_\lambda}. 
\end{equation*}
The fact that the Jack polynomials $P_\lambda$ are homogeneous of degree $|\lambda|$ implies that
\begin{equation*}
	{}_2F_0(-k,\alpha_2;-x/2) = \lim_{\beta_1\rightarrow\infty} {}_2F_1(-k,\alpha_2;\beta_1;-\beta_1 x/2).
\end{equation*}
Substituting $-\beta_1 x/2$ for $x$ in \eqref{HypergDiffEq}, and then taking the limit $\beta_1\rightarrow\infty$, we thus find that 
${}_2F_0(-k,\alpha_2;-x/2)$ satisfies the differential equation
\begin{equation*}
	D_2 F + (\alpha_2 + 1 - k - \kappa(n-1))E_1 F + 2E_0 F = nk\alpha_2 F.
\end{equation*}
It follows that ${}_2F_0(-k,k + a - 1 + \kappa(n-1);-x/2)$ is an eigefunction of $D^B$. We now obtain the statement by comparing leading coefficients of this generalised hypergeometric series and $Y_{(k^n)}$.
\end{proof}

\section{A limit transition from $BC_n$ Jacobi polynomials}
In this section we establish a limit transition from the $BC_n$ Jacobi polynomials to the multivariable Bessel polynomials. We will make use of the representation
\begin{equation}\label{limitRep}
	P^{BC}_\lambda(z) = \prod_{\mu\subset\lambda}\frac{D^{BC} - e^{BC}_\mu}{e^{BC}_\lambda-e^{BC}_\mu}(-4)^{|\lambda|}P_\lambda(t(z))
\end{equation}
for the $BC_n$ Jacobi polynomials $P^{BC}_\lambda$. It is easily verified that the polynomials defined by the right hand side of \eqref{limitRep} indeed satisfy the two defining properties of the $BC_n$ Jacobi polynomials. First of all, it is immediate from \eqref{actionFormulae} that the right hand side of \eqref{limitRep} is triangular in the Jack polynomials $P_\lambda$. In addition, since the operator $\prod_{\mu\subset\lambda}(D^{BC}-e^{BC}_\mu)$ annihilates the subspace spanned by the Jack polynomials $P_\mu$ with $\mu\subset\lambda$, it is also clear that the right hand side of \eqref{limitRep} is an eigenfunction of $D^{BC}$. We mention that this type of representations have been used by Stokman and Koornwinder \cite{SK97}, as well as van Diejen \cite{vD99}, to study limit transitions between different families of multivariable orthogonal polynomials.

It is clear that also the multivariable Bessel polynomials $Y_\lambda$ have a representation of the form \eqref{limitRep}, obtained by substituting $D^B$ for $D^{BC}$, $e^B_\mu$ for $e^{BC}_\mu$, and $P_\lambda(x)$ for $(-4)^{|\lambda|}P_\lambda(t(z))$. As we will show below, there exist a limit transition from the $BC_n$ Jacobi polynomials $P^{BC}_\lambda$ which, in effect, performs these substitutions. Using the fact that
\begin{equation*}
	\coth z = \frac{e^{2z}+1}{e^{2z}-1} = \frac{e^z}{e^z-1} - \frac{1}{e^z+1}
\end{equation*}
we can rewrite the operator $D^{BC}$ as follows:
\begin{multline*}
 D^{BC} = \sum_{i=1}^n \frac{\partial^2}{\partial z_i^2} + \sum_{i=1}^n\left((k_1+2k_2)\frac{e^{z_i}}{e^{z_i}-1} + \left(\frac{k_1}{e^{z_i}-1} - \frac{2k_2}{e^{z_i}+1}\right)\right)\frac{\partial}{\partial z_i}\\ + k_3\sum_{i<j}\left(\frac{e^{z_i}+e^{z_j}}{e^{z_i}-e^{z_j}}\left(\frac{\partial}{\partial z_i} - \frac{\partial}{\partial z_j}\right) + \frac{e^{z_i+z_j}+1}{e^{z_i+z_j}-1}\left(\frac{\partial}{\partial z_i} + \frac{\partial}{\partial z_j}\right)\right).
\end{multline*}
If we set
\begin{equation}\label{kValues}
	k_1 = \frac{1}{2}\left(a - 1 + 2e^\epsilon\right),\quad k_2 = \frac{1}{4}\left(a - 1 - 2e^\epsilon\right),\quad k_3 = \kappa,
\end{equation}
and substitute $(z_1+\epsilon,\ldots,z_n+\epsilon)$ for $(z_1,\ldots,z_n)$, then we obtain the limit
\begin{equation*}
	\lim_{\epsilon\rightarrow\infty}D^{BC} = \sum_{i=1}^n\frac{\partial^2}{\partial z_i^2} + \sum_{i=1}^n\left(a - 1 + 2e^{-z_i}\right)\frac{\partial}{\partial z_i}\\ + 2\kappa\sum_{i\neq j}\frac{e^{z_i}}{e^{z_i}-e^{z_j}}\frac{\partial}{\partial z_i},
\end{equation*}
which is precisely the eigenoperator $D^B$ of the multivariable Bessel polynomials written in the coordinates $(z_1,\ldots,z_n) = (\ln x_1,\ldots,\ln x_n)$. Moreover, we have
\begin{equation*}
	e^{BC}_\mu = e^B_\mu,\quad \lim_{\epsilon\rightarrow\infty}e^{-\epsilon}\sinh^2\frac{z+\epsilon}{2} = \frac{e^z}{4}.
\end{equation*}
By applying the limit in question to \eqref{limitRep} we thus obtain the following limit transition from the $BC_n$ Jacobi polynomials to the multivariable Bessel polynomials:

\begin{proposition}\label{limitProp}
With the parameters $(k_1,k_2,k_3)$ given by \eqref{kValues} we have that
\begin{equation*}
	\lim_{\epsilon\rightarrow\infty} e^{-\epsilon|\lambda|} P^{BC}_\lambda(z_1+\epsilon,\ldots,z_n+\epsilon) = Y_\lambda(e^{z_1},\ldots,e^{z_n}).
\end{equation*}
\end{proposition}

\section{Higher order eigenoperators}
In this section we prove that the multivariable Bessel polynomials are common eigenfunctions of $n$ algebraically independent partial differential operators. We will obtain this result by applying the limit transition in Proposition \ref{limitProp} to a sequence of eigenoperators for the $BC_n$ Jacobi polynomials constructed by Heckman \cite{Hec91}. We remark that Heckman obtained such eigenoperators for the Jacobi polynomials associated with any (integral) root system. In addition, we study the structure of the commutative algebra generated by these higher order eigenoperators of the multivariable Bessel polynomials, and establish an Harish-Chandra type isomorphism onto the algebra of even and permutation invariant polynomials in $n$ variables.

We start by briefly reviewing Heckman's construction, specialised to the root system $BC_n$. In doing so we will essentially follow the exposition in Section 2.3 in Matsuo \cite{Mat92}. The first step consists of inductively defining sequences of partial differential operators $D^{BC}_{d,\pm i}$, where $d\in\mathbb{N}$ and $i = 1,\ldots,n$, by setting $D^{BC}_{0,\pm i} = 1$ and
\begin{multline}\label{DBCDef}
	D^{BC}_{d,\pm i} = \pm\frac{\partial}{\partial z_i}D^{BC}_{d-1,\pm i} \pm \left(\frac{k_1}{2}\coth\frac{1}{2}z_i + k_2\coth z_i\right)\left(D^{BC}_{d-1,\pm i} - D^{BC}_{d-1,\mp i}\right)\\ \pm\frac{k_3}{2}\sum_{j\neq i}\coth\frac{1}{2}(z_i - z_j)\left(D^{BC}_{d-1,\pm i} - D^{BC}_{d-1,\pm j}\right)\\ \pm\frac{k_3}{2}\sum_{j\neq i} \coth\frac{1}{2}(z_i + z_j)\left(D^{BC}_{d-1,\pm i} - D^{BC}_{d-1,\mp j}\right)
\end{multline}
for $d>0$. If we sum over $\pm i$ then we obtain the partial differential operators
\begin{equation*}
	D^{BC}_d = \frac{1}{2}\sum_{\epsilon=\pm 1}\sum_{i=1}^n D^{BC}_{2d,\epsilon i},
\end{equation*}
where we have inserted a factor $1/2$ to simplify certain formulae below. We mention that the corresponding sum of the operators $D^{BC}_{2d+1,\pm i}$ is zero for all $d\in\mathbb{N}$. This is a direct consequence of the fact that $D^{BC}_{d,\mp i} = (-1)^d D^{BC}_{d,\pm i}$, which is easily inferred from the definition above. It was shown by Heckman \cite{Hec91} that the $BC_n$ Jacobi polynomials are common eigenfunctions of the operators $D^{BC}_d$.

We proceed to apply the limit transition to the multivariable Bessel polynomials. Accordingly, we let the parameters $(k_1,k_2,k_3)$ be given by \eqref{kValues} and substitute $(z_1+\epsilon,\ldots,z_n+\epsilon)$ for $(z_1,\ldots,z_n)$ in \eqref{DBCDef}. In analogy with the discussion in Section 4, we then obtain the limit
\begin{multline}\label{DBDef}
	D^B_{d,\pm i} = \pm x_i\frac{\partial}{\partial x_i}D^B_{d-1,\pm i} \pm \frac{1}{2}(a - 1 + 2x_i^{-1})\left(D^B_{d-1,\pm i} - D^B_{d-1,\mp i}\right)\\ \pm\frac{\kappa}{2}\sum_{j\neq i}\frac{x_i+x_j}{x_i-x_j}\left(D^B_{d-1,\pm i} - D^B_{d-1,\pm j}\right) \pm\frac{\kappa}{2}\sum_{j\neq i}\left(D^B_{d-1,\pm i} - D^B_{d-1,\mp j}\right)
\end{multline}
with $D^B_{0,\pm i} = 1$ and, as before, $(x_1,\ldots,x_n) = (e^{z_1},\ldots,e^{z_n})$. The corresponding limit of the operators $D^{BC}_d$ is simply given by
\begin{equation*}
	D^B_d = \frac{1}{2}\sum_{\epsilon=\pm 1}\sum_{i=1}^n D^B_{2d,\epsilon i}.
\end{equation*}
It is a straightforward exercise to verify that $D^B_1 = D^B$, i.e., that the second order eigenoperator $D^B$ of the multivariable Bessel polynomials is the first element in this sequence of operators. Referring to Proposition \ref{limitProp} it is now easy to prove the following:

\begin{theorem}\label{HPDOThm}
The multivariable Bessel polynomials $Y_\lambda$ satisfy a system of differential equations
\begin{equation*}
	D^B_dY_\lambda = e^B_d(\lambda)Y_\lambda,\quad d = 1,2,\ldots,
\end{equation*}
with eigenvalues of the form
\begin{equation*}
	e^B_d(\lambda) = \lambda_1^{2d} +\cdots + \lambda_n^{2d} + \text{l.d.}
\end{equation*}
(where l.d. stands for terms of lower degree). Moreover, the differential operators $D^B_d$ with $d = 1,\ldots,n$ are algebraically independent.
\end{theorem}

\begin{proof}
The fact that the multivariable Bessel polynomials $Y_\lambda$ are eigenfunctions of the operators $D^B_d$ is a direct consequence of Proposition \ref{limitProp} and the discussion above. We observe that, in the order of partial differential operators $x_j\partial/\partial x_j$ with $j = 1,\ldots,n$, the leading term of $D^B_d$ is
\begin{equation*}
	\left(x_1\frac{\partial}{\partial x_1}\right)^{2d} +\cdots + \left(x_n\frac{\partial}{\partial x_n}\right)^{2d}.
\end{equation*}
This clearly implies the stated form of the eigenvalues $e^B_d$. Since the power sums $p_d(x) = x_1^d +\cdots + x_n^d$ with $d = 1,\ldots,n$ are algebraically independent (see e.g.~Section I.2 in Macdonald \cite{Mac95}), we thus obtain the statement.
\end{proof}

\begin{corollary}\label{commutativityCor}
The differential operators $D^B_d$ with $d = 1,\ldots,n$ preserve $\Lambda_n$, and are thus well-defined as linear operators on $\Lambda_n$. As such, they commute with each other.
\end{corollary}

\begin{proof}
We fix a non-negative integer $m$, and consider the linear subspace $\Lambda^{\leq m}_n$. We assume that $\kappa\geq 0$ and that $a$ satisfies the condition \eqref{aInequality}. According to Proposition \ref{nonDegeneracyProp}, the multivariable Bessel polynomials $Y_\lambda$ with $|\lambda|\leq m$ are then well-defined. Moreover, their triangular structure in the Jack polynomials implies that they span $\Lambda^{\leq m}_n$; c.f.~Section I.6 in Macdonald \cite{Mac95}. Under the assumed conditions on $\kappa$ and $a$, the differential operators $D^B_d$ are thus well-defined, and pairwise commutative, as linear operators on $\Lambda^{\leq m}_n$. Since the coefficients of each $D^B_d$ are analytic in $\kappa$ and $a$, this fact immediately extends to arbitrary (complex) parameter values. Since $m$ was chosen arbitrarily, the statement thus follows.
\end{proof}

Theorem \ref{HPDOThm} provides rather little information on the eigenvalues $e^B_d$. In particular, it reveals nothing of their dependence on the parameters $a$ and $\kappa$. Indeed, in general, it seems that this dependence is rather complicated. However, by applying the limit transition in Proposition \ref{limitProp} to results obtained by Heckman \cite{Hec91} (see Theorem 3.11) and Heckman and Opdam \cite{HO87} (see Proposition 2.9) on the Jacobi polynomials associated with root systems, it is relatively easy to establish the existence of eigenoperators of the multivariable Bessel polynomials which have eigenvalues of a very simple form. On the other hand, these eigenoperators do not seem to have a simple explicit expression.

In order to make these remarks precise we introduce the commutative algebra of differential operators $\mathscr{D}_n = \mathbb{C}\lbrack D^B_1,\ldots,D^B_n\rbrack$. We let $\rho^B = (\rho^B_1,\ldots,\rho^B_n)$ be the vector given by
\begin{equation}\label{rhoBVector}
	\rho^B_i = \kappa(n - i) + (a - 1)/2,\quad i = 1,\ldots,n.
\end{equation}
To each $D\in\mathscr{D}_n$ we associate a polynomial $\gamma(D)\in\mathbb{C}\lbrack x_1,\ldots,x_n\rbrack$ by requiring that
\begin{equation}\label{gammaDef}
	D Y_\lambda = \gamma(D)(\lambda+\rho^B)Y_\lambda
\end{equation}
for all partitions $\lambda = (\lambda_1,\ldots,\lambda_n)$. Since $\gamma(D)(\lambda+\rho^B)$ is given by a polynomial in $\lambda$, it is clear that $\gamma(D)$ is indeed uniquely defined by \eqref{gammaDef}. We will use the notation $\mathbb{C}\lbrack x_1,\ldots,x_n\rbrack^{\mathbb{Z}^n_2\rtimes S_n}$ for the algebra of even and permutation symmetric polynomials. Then we have the following:

\begin{lemma}\label{algLemma}
For all $D\in\mathscr{D}_n$,
\begin{equation*}
	\gamma(D)\in\mathbb{C}\lbrack x_1,\ldots,x_n\rbrack^{\mathbb{Z}^n_2\rtimes S_n}.
\end{equation*}
\end{lemma}

\begin{proof}
We will obtain the statement by applying Proposition \ref{limitProp} to the analogous result for the $BC_n$ Jacobi polynomials. We start by recalling this latter result. For that we require the vector $\rho^{BC} = (\rho^{BC}_1,\ldots,\rho^{BC}_n)$ defined by
\begin{equation*}
	\rho^{BC}_i = k_3(n - i) + (k_1 + 2k_2)/2,\quad i = 1,\ldots,n.
\end{equation*}
For each differential operator $D\in\mathbb{C}\lbrack D^{BC}_1,\ldots,D^{BC}_n\rbrack$ we let $\gamma^{BC}(D)\in\mathbb{C}\lbrack x_1,\ldots,x_n\rbrack$ be defined by
\begin{equation}\label{BCGammaDef}
	DP^{BC}_\lambda = \gamma^{BC}(D)(\lambda + \rho^{BC})P^{BC}_\lambda
\end{equation}
for all partitions $\lambda = (\lambda_1,\ldots,\lambda_n)$. Proposition 2.9 in Heckman and Opdam \cite{HO87} then implies that $\gamma^{BC}(D)\in\mathbb{C}\lbrack x_1,\ldots,x_n\rbrack^{\mathbb{Z}^n_2\rtimes S_n}$. We observe that with the parameters $(k_1,k_2,k_3)$ given by \eqref{kValues} we have $\rho^{BC} = \rho^B$. We now obtain the statement as a direct consequence of Proposition \ref{limitProp} and the discussion preceding Theorem \ref{HPDOThm}.
\end{proof}

We can thus define a so-called Harish-Chandra algebra homomorphism
\begin{equation*}
	\gamma: \mathscr{D}_n\rightarrow \mathbb{C}\lbrack x_1,\ldots,x_n\rbrack^{\mathbb{Z}^n_2\rtimes S_n}
\end{equation*}
by letting each $D\in\mathscr{D}_n$ be mapped to $\gamma(D)$. Just as in the case of the $BC_n$ Jacobi polynomials, the algebra homomorphism $\gamma$ thus defined is in fact an isomorphism of algebras. That it is injective is immediate from the fact that the operators $D^B_d$ are algebraically independent, and that it is surjective is a direct consequence of the following analogue of the first part of Theorem 3.11 in Heckman \cite{Hec91}:

\begin{proposition}\label{HPDOProp}
For any polynomial $p\in\mathbb{C}\lbrack x_1,\ldots,x_n\rbrack^{\mathbb{Z}^n_2\rtimes S_n}$ there exists a differential operator $D_p\in\mathscr{D}_n$ such that
\begin{equation*}
	D_pY_\lambda = p(\lambda+\rho^B)Y_\lambda
\end{equation*}
for all partitions $\lambda = (\lambda_1,\ldots,\lambda_n)$.
\end{proposition}

\begin{proof}
We recall that the power sums $p_r(x_1^2,\ldots,x_n^2)$ with $r = 1,\ldots,n$ generate $\mathbb{C}\lbrack x_1,\ldots,x_n\rbrack^{\mathbb{Z}^n_2\rtimes S_n}$. By induction in the degree of the polynomials $p$ the statement thus follows from Theorem \ref{HPDOThm} and Lemma \ref{algLemma}.
\end{proof}

\begin{remark}
It is interesting to note that the differential operators $D_p$ can be constructed in terms of the so-called Cherednik operators. To this end we let $D_i$ denote the Cherednik operator associated with the root system $BC_n$ and the unit vector $e_i$; see e.g.~Cherednik \cite{Che91} or Opdam \cite{Op95} for its definition. We furthermore let $D^{BC}_p$ denote the restriction of $p(D_1,\ldots,D_n)$ to $\mathbb{C}\lbrack x_1,\ldots,x_n\rbrack^{\mathbb{Z}^n_2\rtimes S_n}$. Theorem 2.12 in Opdam \cite{Op95} then implies that $D_p$ can be obtained by applying the limit transition in Proposition \ref{limitProp} to $D^{BC}_p$. The differential operators $D^B_d$ arise in the same way, but with Heckman's 'global' Dunkl operators (see Definition 2.2 in \cite{Hec91}) substituted for the Cherednik operators. The reason why the operators $D_p$ lack the simple recursive structure of the operators $D^{BC}_d$ can now be understood as follows: for a Cherednik operator $D_i$ there exist $w\in\mathbb{Z}^n_2\rtimes S_n$ such that $wD_i\neq D_{w(i)} w$. In the case of the 'global' Dunkl operators, on the other hand, equality always holds, and it is precisely this property which is responsible for the simple recursive structure of the differential operators $D^{BC}_{\pm i}$; see e.g.~Section 2.3 in Matsuo \cite{Mat92}.
\end{remark}

As observed above, we thus have the following:

\begin{corollary}
$\gamma$ is an isomorphism of algebras.
\end{corollary}

In our discussion of orthogonality in Section 7 we shall need the following:

\begin{corollary}\label{separationCor}
Fix a positive integer $m$, and let $\lambda = (\lambda_1,\ldots,\lambda_n)$ and $\mu = (\mu_1,\ldots,\mu_n)$ be two partitions such that $|\lambda|,|\mu|\leq m$ and $\lambda\neq\mu$. Assume that $\kappa\geq 0$ and that $a$ satisfies the condition
\begin{equation*}
	a < -2(m + \kappa(n-1)) + 1.
\end{equation*}
Then there exist $D\in\mathscr{D}_n$ such that $\gamma(D)(\lambda+\rho^B)\neq \gamma(D)(\mu+\rho^B)$.
\end{corollary}

\begin{proof}
By Proposition \ref{HPDOProp} it is sufficient to prove that there exist a polynomial $p\in\mathbb{C}\lbrack x_1,\ldots,x_n\rbrack^{\mathbb{Z}_2^n\rtimes S_n}$ such that $p(\lambda+\rho^B)\neq p(\mu+\rho^B)$. We observe that
\begin{equation*}
	\left(\lambda + \rho^B\right)_n <\cdots < \left(\lambda + \rho^B\right)_1 < 0,
\end{equation*}
and similarly for $\mu$. Since the polynomials in $\mathbb{C}\lbrack x_1,\ldots,x_n\rbrack^{\mathbb{Z}_2^n\rtimes S_n}$ separate points in the open subset of $\mathbb{R}^n$ defined by the inequality $x_n<\cdots<x_1<0$, the statement thus follows.
\end{proof}

\begin{remark}
We note that there are values of $\kappa$ and $a$, and partitions $\lambda$ and $\mu$, for which there exist no $D\in\mathscr{D}_n$ that separates $\lambda$ and $\mu$. For example, with $\kappa = -2$ and $a = 1$ this is the case for $\lambda = 3\delta$ and $\mu = \delta$, where $\delta = (n-1,n-2,\ldots,0)$. However, it is clear that such examples are rather rare.
\end{remark}

\section{Recurrence relations}
It is well known that the one-variable Bessel polynomials $y_n$ satisfy the three-term recurrence relation
\begin{multline*}
	\left((2n + a)(2n + a - 2)\frac{x}{2} + a - 2\right)(2n + a - 1) y_n(x)\\ = (n + a - 1)(2n + a - 2)y_{n+1}(x) - n(2n + a)y_{n-1}(x)
\end{multline*}
when normalised such that
\begin{equation*}
	y_n(x) = \frac{(n+a-1)_n}{2^n}Y_{(n)}(x);
\end{equation*}
see e.g.~Krall and Frink \cite{KF49}. In this section we will obtain a natural multivariable analogue of this recurrence relation: the explicit expansion of all products between an elementary symmetric polynomial and a multivariable Bessel polynomial in terms of the multivariable Bessel polynomials themselves. We mention that this type of recurrence relations are often referred to as Pieri type formulae. Since the elementary symmetric functions generate $\Lambda_n$ (see e.g.~Section I.2 in Macdonald \cite{Mac95}), this enables us to obtain the expansion in multivariable Bessel polynomials of the product between any symmetric polynomial and a multivariable Bessel polynomial. In deducing these recurrence relations we will start with the simplest case, corresponding to the elementary symmetric polynomial of degree one. We will sketch a proof of this recurrence relation, based on the orthogonality of the multivariable Bessel polynomials, as deduced in Section 7, and their explicit series expansion in Jack polynomials, as obtained in Proposition \ref{seriesProp}. We will then proceed to treat the general case using recurrence relations obtained by van Diejen \cite{vD99} for the $BC_n$ Jacobi polynomials, and the limit transition to the multivariable Bessel polynomials established in Proposition \ref{limitProp}. In order to simplify a comparison with his results on the $BC_n$ Jacobi case, we will to a large extent make use of the same notation as van Diejen.

As will become apparent below, it will be convenient to employ a normalisation of the multivariable Bessel polynomials different from the monic $Y_\lambda$. This particular normalisation has a simple expression in terms of the functions
\begin{equation*}
	\hat{\Delta}^B_{\pm}(z) = \prod_{i<j}\hat{d}^B_{v,\pm}(z_i - z_j)\hat{d}^B_{v,\pm}(z_i + z_j)\prod_{i=1}^n \hat{d}^B_{w,\pm}(z_i),
\end{equation*}
where
\begin{align*}
	\hat{d}^B_{v,\pm}(z) &= \frac{\Gamma(\pm\kappa+z+s_\pm)}{\Gamma(z+s_\pm)},\\
	\hat{d}^B_{w,\pm}(z) &= \frac{\Gamma(\pm(a-1)/2+z+s_\pm)}{\Gamma(2z+s_\pm)},
\end{align*}
with $s_+ = 0$ and $s_- = 1$. Indeed, we can now define 'renormalised' multivariable Bessel polynomials by
\begin{equation}\label{tildeNormalisation}
	\tilde{Y}_\lambda = 2^{-|\lambda|}\frac{\hat{\Delta}^B_+(\rho^B)}{\hat{\Delta}^B_+(\rho^B+\lambda)} Y_\lambda,
\end{equation}
where the vector $\rho^B$ is given by \eqref{rhoBVector}. In the one-variable case it is straightforward to verify that this particular normalisation of the Bessel polynomials coincide with the normalisation employed by Krall and Frink \cite{KF49}, i.e., $\tilde{Y}_{(n)}(x) = y_n(x)$. It is also interesting to note that the normalisation in question has a simple characterisation in terms of a specialisation formula.

\begin{proposition}\label{specialisationProp}
The multivariable Bessel polynomials $\tilde{Y}_\lambda$ are such that
\begin{equation}\label{specialisationFormula}
	\tilde{Y}_\lambda(0^n) = 1,
\end{equation}
i.e., the constant term of $\tilde{Y}_\lambda$ equals one.
\end{proposition}

In order to prove the statement we will make use of the limit transition from the $BC_n$ Jacobi polynomials, and a known specialisation formula for these polynomials. To recall this formula we let $\hat{\Delta}^{BC}_\pm$ denote the functions obtained from $\hat{\Delta}^B_\pm$ upon substituting
\begin{align*}
	\hat{d}^{BC}_{v,\pm}(z) &= \frac{\Gamma(\pm k_3+z+s_\pm)}{\Gamma(z+s_\pm)},\\ \hat{d}^{BC}_{w,\pm}(z) &= \frac{\Gamma(\pm(k_1+2k_2)/2+z+s_\pm)\Gamma(\pm(k_1+1)/2+z+s_\pm)}{\Gamma(2z+s_\pm)}
\end{align*}
for $\hat{d}^B_{v,\pm}$ and $\hat{d}^B_{w,\pm}$, respectively, where, as before, $s_+ = 0$ and $s_- = 1$. It has been shown by Opdam \cite{Op89} (see Corollary 5.2 for $R=BC_n$) that
\begin{equation*}
	P^{BC}_\lambda(0^n) = 2^{2|\lambda|}\frac{\hat{\Delta}^{BC}_+(\rho^{BC}+\lambda)}{\hat{\Delta}^{BC}_+(\rho^{BC})}
\end{equation*}
with the vector $\rho^{BC} = (\rho^{BC}_1,\ldots,\rho^{BC}_n)$ given by
\begin{equation*}
	\rho^{BC}_i = k_3(n - i) + (k_1 + 2k_2)/2,\quad i = 1,\ldots,n.
\end{equation*}
We remark that Opdam formulated his result in terms of a function $\tilde{c}$ which (in the case of the root system $BC_n$) is defined somewhat differently from $\hat{\Delta}^{BC}_+$. However, using the duplication formula $\Gamma(z)\Gamma(z+1/2) = \Gamma(2z)\sqrt{\pi}/2^{2z-1}$ for the Gamma function it is a straightforward exercise to verify that the two formulations are equivalent.

\begin{proof}[Proof of Proposition \ref{specialisationProp}]
We let the values of the parameters $(k_1,k_2,k_3)$ be specified by \eqref{kValues}. By definition,
\begin{equation*}
	P^{BC}_\lambda(z) = \sum_{\mu\subseteq\lambda}u_{\lambda\mu}P_\mu(t(z))
\end{equation*}
for some coefficients $u_{\lambda\mu}$. Since $\lim_{\epsilon\rightarrow\infty}e^{-\epsilon}t(z+\epsilon) = -e^z/4$, we have
\begin{equation*}
	\lim_{\epsilon\rightarrow\infty} e^{-\epsilon|\lambda|} P^{BC}_\lambda(z+\epsilon) = \sum_{\mu\subseteq\lambda}(-4)^{-|\mu|}\left(\lim_{\epsilon\rightarrow\infty} e^{-\epsilon(|\lambda|-|\mu|)}u_{\lambda\mu}\right)P_\mu(e^z).
\end{equation*}
It follows from Proposition \ref{limitProp} that
\begin{equation*}
	Y_\lambda(0^n) = \lim_{\epsilon\rightarrow\infty} e^{-\epsilon|\lambda|}u_{\lambda,(0)} = \lim_{\epsilon\rightarrow\infty} e^{-\epsilon|\lambda|}P^{BC}_\lambda(0^n).
\end{equation*}
Using the difference equation $\Gamma(z+1) = z\Gamma(z)$ it is straightforward to verify that
\begin{equation*}
	\lim_{\epsilon\rightarrow\infty} e^{-\epsilon} \frac{\hat{d}^{BC}_{w,+}(z+1)}{\hat{d}^{BC}_{w,+}(z)} = \frac{1}{2}\frac{\hat{d}^B_{w,+}(z+1)}{\hat{d}^B_{w,+}(z)}.
\end{equation*}
In addition, $\rho^{BC} = \rho^B$. It follows that
\begin{equation*}
	\lim_{\epsilon\rightarrow\infty} e^{-\epsilon|\lambda|}2^{2|\lambda|}\frac{\hat{\Delta}^{BC}_+(\rho^{BC}+\lambda)}{\hat{\Delta}^{BC}_+(\rho^{BC})} = 2^{|\lambda|}\frac{\hat{\Delta}^B_+(\rho^B+\lambda)}{\hat{\Delta}^B_+(\rho^B)}
\end{equation*}
which clearly implies the statement.
\end{proof}

We proceed to deduce the recurrence relations in question for the multivariable Bessel polynomials. As we will see below, these recurrence relations have a simple expression in terms of the functions
\begin{equation*}
	\hat{v}^B(z) = \frac{\kappa+z}{z},\quad \hat{w}^B(z) = \frac{((a-1)/2+z)}{2z(2z+1)}.
\end{equation*}
It is readily inferred from the difference equation $\Gamma(z+1) = z\Gamma(z)$ for the Gamma function that these functions appear as the coefficients in the following difference equations for the functions $\hat{d}^B_{v,\pm}$ and $\hat{d}^B_{w,\pm}$:

\begin{lemma}\label{differenceEqsLemma}
The functions $\hat{d}^B_{v,\pm}$ and $\hat{d}^B_{w,\pm}$ satisfy the difference equations
\begin{align*}
	\hat{d}^B_{v,+}(z + 1) &= \hat{v}^B(z)\hat{d}^B_{v,+}(z), & \hat{d}^B_{v,-}(z + 1) &= \hat{v}^B(-z-1)\hat{d}^B_{v,-}(z),\\
	\hat{d}^B_{w,+}(z + 1) &= \hat{w}^B(z)\hat{d}^B_{w,+}(z), & \hat{d}^B_{w,-}(z + 1) &= -\hat{w}^B(-z-1)\hat{d}^B_{w,-}(z).
\end{align*}
\end{lemma}

We are now ready to state and prove the simplest recurrence relation for the multivariable Bessel polynomials, corresponding to the elementary symmetric polynomial of degree one.

\begin{proposition}\label{recurRelProp}
The multivariable Bessel polynomials $\tilde{Y}_\lambda$ satisfy the recurrence relation
\begin{equation}\label{simplestRecurRel}
\begin{split}
	\frac{1}{2}(x_1 +\cdots + x_n)\tilde{Y}_\lambda(x) &= \sum_i \hat{V}_i\left(\rho^B + \lambda\right)\left(\tilde{Y}_{\lambda^{(i)}}(x) - \tilde{Y}_\lambda(x)\right)\\ &\quad + \sum_i \hat{V}_{-i}\left(\rho^B + \lambda\right)\left(\tilde{Y}_{\lambda_{(i)}}(x) - \tilde{Y}_\lambda(x)\right)
\end{split}
\end{equation}
with
\begin{equation}\label{simplestRecurRelCoeffs}
	\hat{V}_{\pm i}(z) = \hat{w}^B(\pm z_i)\prod_{j\neq i} \hat{v}^B(\pm z_i + z_j)\hat{v}^B(\pm z_i - z_j),
\end{equation}
where the first and second sum is over all $i = 1,\ldots,n$ such that $\lambda^{(i)}$ and $\lambda_{(i)}$ is a partition, respectively.
\end{proposition}

\begin{proof}
We fix a partition $\lambda = (\lambda_1,\ldots,\lambda_n)$. In order to make use of the orthogonality of the multivariable Bessel polynomials, as established in Theorem \ref{orthogonalityThm}, we assume that $\kappa\geq 0$ and that the condition \eqref{aCond} is satisfied for $m = |\lambda| + 1$. It follows from \eqref{p1Action} and condition (1) in Definition \ref{besselDef} that
\begin{equation}\label{BesselExp}
	\frac{1}{2}(x_1 +\cdots +x_n)\tilde{Y}_\lambda(x) = \sum_{i=1}^n\hat{V}_i\tilde{Y}_{\lambda^{(i)}}(x) + \hat{V}_0\tilde{Y}_\lambda(x) + \sum_{i=1}^n \hat{V}_{-i}\tilde{Y}_{\lambda_{(i)}}(x) + P(x)
\end{equation}
for some coefficients $\hat{V}_i$ and a symmetric polynomial $P$ of degree less than $|\lambda|-1$. We let $\mu = (\mu_1,\ldots,\mu_n)$ be a partition such that $|\mu|\leq|\lambda|-2$. Then, $(x_1+\cdots+x_n)\tilde{Y}_\mu(x)$ is a linear combination of multivariable Bessel polynomials $\tilde{Y}_\nu(x)$ such that $|\nu|<|\lambda|$. It thus follows from Theorem \ref{orthogonalityThm} that
\begin{equation*}
	0 = \left\langle (x_1+\cdots+x_n)\tilde{Y}_\lambda,\tilde{Y}_\mu\right\rangle_{a,\kappa} = \left\langle P,\tilde{Y}_\mu\right\rangle_{a,\kappa}.
\end{equation*}
Since the multivariable Bessel polynomials $\tilde{Y}_\mu$ with $|\mu|\leq|\lambda|-2$ span $\Lambda_n^{\leq|\lambda|-2}$ (c.f.~the proof of Corollary \ref{commutativityCor}), this implies that $P = 0$. Moreover, by setting $x = (0^n)$ in \eqref{BesselExp} and using the specialisation formula \eqref{specialisationFormula} we find that
\begin{equation*}
	\hat{V}_0 = -\sum_{i=1}^n\left(\hat{V}_i + \hat{V}_{-i}\right).
\end{equation*}
Hence, the expansion of $(x_1 +\cdots + x_n)\tilde{Y}_\lambda(x)$ in 'renormalised' Bessel polynomials $\tilde{Y}_\mu$ is indeed of the form \eqref{simplestRecurRel}. There remains to compute the coefficients $\hat{V}_{\pm i}$ for $i = 1,\ldots,n$. It can be inferred from (6.24) in Chapter VI of Macdonald's book \cite{Mac95} that
\begin{equation*}
	(x_1 +\cdots + x_n)P_\lambda(x) = \sum_i  \psi^\prime_{\lambda^{(i)}/\lambda}P_{\lambda^{(i)}}(x)
\end{equation*}
with
\begin{equation*}
	\psi^\prime_{\lambda^{(i)}/\lambda} = \prod_{j<i}\frac{\kappa(j-i+1)+\lambda_i-\lambda_j}{\kappa(j-i)+\lambda_i-\lambda_j}\frac{\kappa(j-i-1)+\lambda_i-\lambda_j+1}{\kappa(j-i)+\lambda_i-\lambda_j+1}.
\end{equation*}
Applying this recurrence relation to the left hand side of \eqref{simplestRecurRel}, and then equating coefficients with the right hand side we obtain
\begin{equation*}
	\hat{V}_i = \frac{\hat{\Delta}^B_+(\rho^B+\lambda^{(i)})}{\hat{\Delta}^B_+(\rho^B+\lambda)}\psi^\prime_{\lambda^{(i)}/\lambda}
\end{equation*}
for $i = 1,\ldots,n$. Using Lemma \ref{differenceEqsLemma}, and the formula for $\psi^\prime_{\lambda^{(i)}/\lambda}$ stated above, it is now straightforward to verify that the coefficients $\hat{V}_i$ indeed are given by \eqref{simplestRecurRelCoeffs} for $i = 1,\ldots,n$. That \eqref{simplestRecurRelCoeffs} is valid also for $i = -1,\ldots,-n$ can be verified in a similar manner, using the series expansion of the multivariable Bessel polynomials obtained in Proposition \ref{seriesProp}. This proves the statement under the assumed conditions on the parameters $a$ and $\kappa$. However, by
analytic continuation in these parameters the proof immediately extends to any values of $a$ and $\kappa$ for which the multivariable Bessel polynomial in question is well-defined.
\end{proof}

In principle we could use the same method as in the proof of Proposition \ref{recurRelProp} to deduce higher order recurrence relations for the multivariable Bessel polynomials. However, it is clear that this would lead to lengthy and complicated computations. We will therefore use a different approach. To this end we recall that, by exploiting limit transitions from the multivariable Askey-Wilson polynomials, van Diejen \cite{vD99} obtained recurrence relations for a number of different families of orthogonal polynomials. Among them are the $BC_n$ Jacobi polynomials. As we will show below, by applying the limit transition obtained in Proposition \ref{limitProp} to these latter recurrence relations it is simple and straightforward to obtain the remaining recurrence relations for the multivariable Bessel polynomials. We remark that van Diejen used a definition of the $BC_n$ Jacobi polynomials -- in his notation denoted $p^J_\lambda$ -- which is somewhat different from the one employed in this paper. It is readily verified that the precise relation between these two definitions is given by
\begin{equation*}
	p^J_\lambda(ix_1,\ldots,ix_n;\nu_0,\nu_1,\nu) = P^{BC}_\lambda(x_1,\ldots,x_n;k_1,k_2,k_3)
\end{equation*}
with the parameters $(\nu_0,\nu_1,\nu) = (k_1+k_2,k_2,k_3)$, and where $i$ here denotes the imaginary unit. In order to avoid the use of more than one definition of the $BC_n$ Jacobi polynomials, we will formulate van Diejen's results in terms of the polynomials $P^{BC}_\lambda$. For that we require the 'renormalised' $BC_n$ Jacobi polynomials
\begin{equation*}
		\tilde{P}^{BC}_\lambda = 2^{-2|\lambda|}\frac{\hat{\Delta}^{BC}_+(\rho^{BC})}{\hat{\Delta}^{BC}_+(\rho^{BC}+\lambda)}P^{BC}_\lambda.
\end{equation*}
It is clear from the discussion above that, as in the case of the multivariable Bessel polynomials, these 'renormalised' $BC_n$ Jacobi polynomials can be characterised by the fact that their constant term equals one. Furthermore, the functions
\begin{equation*}
	\hat{v}^{BC}(z) = \frac{\kappa+z}{z},\quad \hat{w}^{BC}(z) = \frac{((k_1+2k_2)/2+z)((k_1+1)/2+z)}{2z(2z+1)},
\end{equation*}
and the symmetric trigonometric polynomials
\begin{equation*}
	\hat{E}^{BC}_r(z) = (-1)^{r+1}\sum_{\substack{I\subset\lbrace 1,\ldots,n\rbrace\\ |I|=r}}\prod_{i\in I}\sinh^2\frac{z_i}{2},
\end{equation*}
will be required. We note, in particular, that the trigonometric polynomials $\hat{E}_r$ are, up to a difference in sign for odd $r$, the elementary symmetric polynomials in the functions $t_i(z_i) = -\sinh^2\frac{z_i}{2}$ with $i = 1,\ldots,n$. Given an index set $I\subset\lbrace 1,\ldots,n\rbrace$ and a configuration of signs $\epsilon_i = \pm 1$, $i\in I$, we will make use of the short-hand notation
\begin{equation*}
	e_{\epsilon I} = \sum_{i\in I}\epsilon_i e_i.
\end{equation*}
For each $r = 1,\ldots, n$ van Diejen \cite{vD99} (see Theorem 6.4) obtained the following reccurence relation for the $BC_n$ Jacobi polynomials:
\begin{equation}\label{JacobiRecurRels}
	\hat{E}^{BC}_r(z) \tilde{P}^{BC}_\lambda(z) = \sum_{I,\epsilon}\hat{U}_{I^c,r-|I|}\left(\rho^{BC}+\lambda\right)\hat{V}_{\epsilon I,I^c}\left(\rho^{BC}+\lambda\right)\tilde{P}^{BC}_{\lambda+e_{\epsilon I}}(z),
\end{equation}
where the sum is over all index sets $I\subset\lbrace 1,\ldots,n\rbrace$ with $0\leq|I|\leq r$ and configuration of signs $\epsilon_i = \pm 1$, $i\in I$, such that $\lambda + e_{\epsilon I}$ is a partition, and where
\begin{align*}
	\hat{U}_{J,p}(z) &= (-1)^p\sum_{\substack{K\subset J\\ |K|=p}}\sum_{\substack{\epsilon_k=\pm 1\\ k\in K}} \Bigg(\prod_{k\in K}\hat{w}^{BC}(\epsilon_k z_k)\\ &\quad\times\prod_{\substack{k,k^\prime\in K\\ k<k^\prime}}\hat{v}^{BC}(\epsilon_k z_k + \epsilon_{k^\prime}z_{k^\prime})\hat{v}^{BC}(-\epsilon_k z_k - \epsilon_{k^\prime}z_{k^\prime} - 1)\\ &\quad \times \prod_{\substack{k\in K\\ j\in J\setminus K}}\hat{v}^{BC}(\epsilon_k z_k+z_j)\hat{v}^{BC}(\epsilon_k z_k-z_j)\Bigg),\\
	\hat{V}_{\epsilon I,J}(z) &= \prod_{i\in I}\hat{w}^{BC}(\epsilon_i z_i)\prod_{\substack{i,i^\prime\in I\\ i<i^\prime}}\hat{v}^{BC}(\epsilon_i z_i + \epsilon_{i^\prime}z_{i^\prime})\hat{v}^{BC}(\epsilon_i z_i + \epsilon_{i^\prime}z_{i^\prime} + 1)\\ &\quad \times \prod_{\substack{i\in I\\ j\in J}}\hat{v}^{BC}(\epsilon_i z_i + z_j)\hat{v}^{BC}(\epsilon_i z_i - z_j).
\end{align*}

The recurrence relations for the multivariable Bessel polynomials will have the same structure as \eqref{JacobiRecurRels}. We only have to make the appropriate substitutions. More precisely, with
\begin{equation*}
	\hat{E}^B_r(x) = \frac{(-1)^{r+1}}{2^r}\sum_{\substack{I\subset\lbrace 1,\ldots,n\rbrace\\ |I|=r}}\prod_{i\in I}x_i
\end{equation*}
we have the following:

\begin{theorem}\label{RecurRelsThm}
For each $r = 1,\ldots,n$ the 'renormalised' multivariable Bessel polynomials $\tilde{Y}_\lambda$ satisfy a recurrence relation of the same form as \eqref{JacobiRecurRels}, but with $\hat{E}^B_r$ substituted for $\hat{E}^{BC}_r$, $\rho^B$ for $\rho^{BC}$, and the functions $\hat{v}^{BC}$ and $\hat{w}^{BC}$ replaced by $\hat{v}^B$ and $\hat{w}^B$, respectively.
\end{theorem}

\begin{proof}
We let the values of the parameters $(k_1,k_2,k_3)$ be given by \eqref{kValues}. Then, $\hat{v}^{BC} = \hat{v}^B$ and $\rho^{BC} = \rho^B$. In addition, it is readily verified that
\begin{align*}
	\lim_{\epsilon\rightarrow\infty} e^{-\epsilon}\hat{w}^{BC}(z) &= \frac{1}{2}\hat{w}^B(z),\\
	\lim_{\epsilon\rightarrow\infty} e^{-\epsilon r}\hat{E}^{BC}_r(z_1,\ldots,z_n) &= \frac{1}{2^{2r}}\hat{E}^B_r(e^{z_1},\ldots,e^{z_n}).
\end{align*}
It follows from Proposition \ref{limitProp} and the proof of Proposition \ref{specialisationProp} that
\begin{equation*}
	\lim_{\epsilon\rightarrow\infty} \tilde{P}^{BC}_\lambda(z_1+\epsilon,\ldots,z_n+\epsilon) = \tilde{Y}_\lambda(e^{z_1},\ldots,e^{z_n}).
\end{equation*}
We thus obtain the statement by first multiplying both sides of \eqref{JacobiRecurRels} by $2^re^{-\epsilon r}$ and then taking the limit $\epsilon\rightarrow\infty$.
\end{proof}

\section{Orthogonality and norms}
In the one-variable case it is well known that there exist no (positive) measure on the real line with respect to which all Bessel polynomials are orthogonal; see e.g.~Section 18 in Krall and Frink \cite{KF49}. However, for a fixed maximum degree, and sufficiently negative values of $a$, the corresponding Bessel polynomials can be shown to be orthogonal with respect to just such an inner product. In this section we will establish a similar statement for the multivariable Bessel polynomials. More precisely, under the assumption that $\kappa\geq 0$, we will prove the following: given any non-negative integer $m$, and sufficiently negative $a$, the Bessel polynomials $Y_\lambda$ of degree at most $m$ form an orthogonal system with respect to the inner product in the Hilbert space $L^2(\mathbb{R}^n_+,d\mu^{(B)})$ with
\begin{equation*}
	d\mu^{(B)}(x;a,\kappa) = \prod_{i=1}^n x_i^{a-2}e^{-2/x_i}\prod_{i<j}|x_i - x_j|^{2\kappa}dx_1\cdots dx_n.
\end{equation*}
Suppose now that we apply the Gram-Schmidt orthogonalisation procedure to this inner product to obtain a set of symmetric polynomials labelled by the partitions $\lambda = (\lambda_1,\ldots,\lambda_n)$ of weight $|\lambda|\leq m$. It is then clear that, with the appropriate normalisation, these symmetric polynomials will coincide with the multivariable Bessel polynomials for the parameter values mentioned above. Moreover, since the multivariable Bessel polynomials are rational functions in $a$ and $\kappa$ (see Corollary \ref{analyticCorollary}) this fact can be immediately extended to generic parameter values.

At this point we fix a non-negative integer $m$, and consider the resulting subspace $\Lambda_n^{\leq m}\subset\Lambda_n$, as defined in Section 2.1. We let $\langle \cdot,\cdot\rangle_{a,\kappa}$ denote the inner product in $L^2(\mathbb{R}^n_+,d\mu^{(B)})$, i.e.,
\begin{equation}\label{innerProd}
	\langle f,g\rangle_{a,\kappa} = \int_{\mathbb{R}^n_+}f(x)\overline{g(x)}d\mu^{(B)}(x;a,\kappa)
\end{equation}
for any $f,g\in L^2(\mathbb{R}^n_+,d\mu^{(B)})$. It is clear that whether or not $\Lambda^{\leq m}_n$ is contained in this Hilbert space depends on the values of the parameters $a$ and $\kappa$. A precise condition is given by the following:

\begin{lemma}\label{L2Lemma}
Assume that $\kappa\geq 0$. Then $\Lambda_n^{\leq m}\subset L^2(\mathbb{R}^n_+,d\mu^{(B)})$ if and only if\begin{equation}\label{aCond}
	a < -2(m + \kappa(n-1)) + 1.
\end{equation}
\end{lemma}

\begin{proof}
We fix two partitions $\lambda = (\lambda_1,\ldots,\lambda_n)$ and $\mu = (\mu_1,\ldots,\mu_n)$ such that $|\lambda|, |\mu|\leq m$, and consider the integral
\begin{equation}\label{mIntegral}
	\int_{\mathbb{R}^n_+} m_\lambda m_\mu d\mu^{(B)}.
\end{equation}
Since $d\mu^{(B)}(a,\kappa;x)$ is symmetric in $x$, we can replace the domain of integration by the open subset in $\mathbb{R}^n_+$ defined by the inequality
\begin{equation*}
	0 < x_n < x_{n-1} <\cdots < x_1.
\end{equation*}
We observe that
\begin{equation*}
	d\mu^{(B)}(a,\kappa;x) = \prod_{i=1}^n |x_i|^{a-2+2\kappa(n-i)} e^{-\frac{2}{x_i}}\prod_{i<j}\left|1 - \frac{x_j}{x_i}\right|^{2\kappa} dx_1\cdots dx_n.
\end{equation*}
It follows that
\begin{multline*}
	\int_{0<x_n<\cdots <x_1}m_\lambda(x) m_\mu(x) d\mu^{(B)}(a,\kappa;x)\\ < \int_{0<x_n<\cdots <x_1} m_\lambda(x) m_\mu(x)\prod_{i=1}^n |x_i|^{a-2+2\kappa(n-i)} e^{-\frac{2}{x_i}} dx_1\cdots dx_n.
\end{multline*}
Clearly, the latter integral exists if and only if
\begin{equation*}
	\lambda_1 + \mu_1 + a - 2 + 2\kappa(n-1) < -1.
\end{equation*}
By setting $\lambda = \mu = (m)$ we thus conclude that $\Lambda^{\leq m}_n\subset L^2(\mathbb{R}^n_+,d\mu^{(B)})$ if $a$ satisfies \eqref{aCond}. We now suppose that $a$ does not satisfy the condition \eqref{aCond}. We keep $\lambda = \mu = (m)$, and fix all variables $x$ except $x_1$. It is then readily seen that the integrand of \eqref{mIntegral} can not be integrated in $x_1$. Consequently, the integral does not exist.
\end{proof}

In order to establish that the multivariable Bessel polynomials contained in $\Lambda^{\leq m}_n$ form an orthogonal system we proceed to prove that their eigenoperators, obtained in Section 5, are symmetric operators on $\Lambda^{\leq m}_n$ with respect to the inner product $\langle\cdot,\cdot\rangle_{a,\kappa}$.

\begin{lemma}\label{symmetryLemma}
Assume that $\kappa\geq 0$ and that $a$ satisfies the condition \eqref{aCond}. Then
\begin{equation*}
	\left\langle D^B_d f,g\right\rangle_{a,\kappa} = \left\langle f,D^B_d g\right\rangle_{a,\kappa}
\end{equation*}
for all $f,g\in\Lambda^{\leq m}_n$ and $d = 1,\ldots,n$.
\end{lemma}

\begin{proof}
We let
\begin{equation*}
	W(x;a,\kappa) = \prod_{i=1}^n x_i^{a-2}e^{-2/x_i}\prod_{i<j}|x_i - x_j|^{2\kappa}
\end{equation*}
so that $d\mu^{(B)}(a,\kappa;x) = W(x;a,\kappa)dx_1\cdots dx_n$. We will prove the statement by verifying that
\begin{equation}\label{equivToSym}
	\sum_{\epsilon=\pm 1}\sum_{i=1}^n\left\langle D^B_{p,\epsilon i} f,D^B_{q,\epsilon i} g\right\rangle_{a,\kappa} = -\sum_{\epsilon=\pm 1}\sum_{i=1}^n\left\langle D^B_{p-1,\epsilon i} f,D^B_{q+1,\epsilon i} g\right\rangle_{a,\kappa}
\end{equation}
for all $f,g\in\Lambda_n^{\leq m}$, and positive and non-negative integers $p$ and $q$, respectively. By partial integration in $x_i$ we deduce from \eqref{innerProd} that
\begin{multline*}
	\left\langle x_i\frac{\partial}{\partial x_i}D^B_{p-1,i}f,D^B_{q,i}g\right\rangle_{a,\kappa}\\ = -\left\langle D^B_{p-1,i}f,x_i\frac{\partial}{\partial x_i}D^B_{q,i}g\right\rangle_{a,\kappa} - \left\langle D^B_{p-1,i}f,\left(W^{-1}\frac{\partial}{\partial x_i}x_iW\right)D^B_{q,i}g\right\rangle_{a,\kappa}.
\end{multline*}
It is readily verified that
\begin{equation*}
	W^{-1}\frac{\partial}{\partial x_i}x_iW = a -1 + \frac{2}{x_i} + \kappa(n-1)  + \kappa\sum_{j\neq i}\frac{x_i+x_j}{x_i - x_j}.
\end{equation*}
The definition, as stated in \eqref{DBDef}, of the differential operators $D^B_d$ thus imply that
\begin{equation*}
\begin{split}
	\left\langle D^B_{p,\epsilon i}f,D^B_{q,\epsilon i}g\right\rangle_{a,\kappa} &= -\left\langle D^B_{p-1,\epsilon i}f,\epsilon x_i\frac{\partial}{\partial x_i}D^B_{q,\epsilon i}g\right\rangle_{a,\kappa}\\ &\quad - \left\langle (D^B_{p-1,\epsilon i} + D^B_{p-1,-\epsilon i})f,\frac{\epsilon}{2}\left(a-1+\frac{2}{x_i}\right)D^B_{q,\epsilon i}g\right\rangle_{a,\kappa}\\ &\quad - \sum_{j\neq i}\left\langle (D^B_{p-1,\epsilon i} + D^B_{p-1,\epsilon j})f,\epsilon\frac{\kappa}{2}\frac{x_i+x_j}{x_i-x_j} D^B_{q,\epsilon i}g\right\rangle_{a,\kappa}\\ &\quad -\sum_{j\neq i}\left\langle (D^B_{p-1,\epsilon i} + D^B_{p-1,-\epsilon j})f,\epsilon\frac{\kappa}{2}D^B_{q,\epsilon i}\right\rangle_{a,\kappa}.
\end{split}
\end{equation*}
Taking the sum over $\epsilon$ and $i$, and using the invariance under the interchange of the summation indices $i$ and $j$, as well as under the substitution of $-\epsilon$ for $\epsilon$, we obtain \eqref{equivToSym}.
\end{proof}

To recapitulate, we know from Theorem \ref{HPDOThm} that the multivariable Bessel polynomials are common eigenfunctions of the differential operators $D^B_d$ for $d = 1,\ldots,n$. In addition, we have just shown that these operators, for appropriate values of $\kappa$ and $a$,  are symmetric on $\Lambda^{\leq m}_n$ with respect to the inner product $\langle\cdot,\cdot\rangle_{a,\kappa}$. Since the differential operators in the algebra $\mathscr{D}_n = \mathbb{C}\lbrack D^B_1,\ldots D^B_n\rbrack$ separate the multivariable Bessel polynomials in question, as shown in Corollary \ref{separationCor}, we thus obtain the following orthogonality result:

\begin{theorem}\label{orthogonalityThm}
Assume that $\kappa\geq 0$ and that $a$ satisfies the condition \eqref{aCond}. Then the multivariable Bessel polynomials $Y_\lambda$ with $|\lambda|\leq m$ form an orthogonal system with respect to the inner product $\langle\cdot,\cdot\rangle_{a,\kappa}$, i.e., for any two partitions $\lambda = (\lambda_1,\ldots,\lambda_n)$ and $\mu = (\mu_1,\ldots,\mu_n)$ such that $|\lambda|,|\mu|\leq m$,
\begin{equation*}
	\langle Y_\lambda, Y_\mu\rangle_{a,\kappa} = 0,\quad \text{if}~\lambda\neq\mu.
\end{equation*}
\end{theorem}

We proceed to compute the squared norms
\begin{equation*}
	N_\lambda(a,\kappa) = \langle Y_\lambda, Y_\lambda\rangle_{a,\kappa},
\end{equation*}
thus enabling us to convert the multivariable Bessel polynomials contained in $\Lambda^{\leq m}_n$ into an orthonormal system. We will proceed in two steps: first we use the recurrence relations obtained in Section 6 to reduce the problem to that of computing $N_{(0)}$, this latter normalisation factor is then computed using a well-known limiting case of an integral formula due Selberg \cite{Sel44}; see also Forrester and Warnaar \cite{FW07}. The first step is accomplished in the following:

\begin{lemma}\label{NormFactorQuotientLemma}
Assume that $\kappa\geq 0$ and that $a$ satisfies the condition \eqref{aCond}. Let $\lambda = (\lambda_1,\ldots,\lambda_n)$ be a partition such that $|\lambda|\leq m$. Then
\begin{equation}\label{NormFactorQuotient}
	\frac{N_\lambda}{N_{(0)}} = (-4)^{|\lambda|}\frac{\hat{\Delta}^B_+(\rho^B+\lambda)\hat{\Delta}^B_-(\rho^B+\lambda)}{\hat{\Delta}^B_+(\rho^B)\hat{\Delta}^B_-(\rho^B)}.
\end{equation}
\end{lemma}

\begin{proof}
We will prove the following equivalent statement: the squared norms
\begin{equation*}
	\tilde{N}_\mu = \left\langle \tilde{Y}_\mu, \tilde{Y}_\mu\right\rangle_{a,\kappa}
\end{equation*}
of the 'renormalised' multivariable Bessel polynomials $\tilde{Y}_\mu$ with $|\mu|\leq m$, as defined in \eqref{tildeNormalisation}, are given by
\begin{equation}\label{TildeNormFactorQuotient}
	\frac{\tilde{N}_\mu}{\tilde{N}_{(0)}} = (-1)^{|\mu|}\frac{\hat{\Delta}^B_+(\rho^B)\hat{\Delta}^B_-(\rho^B+\mu)}{\hat{\Delta}^B_-(\rho^B)\hat{\Delta}^B_+(\rho^B+\mu)}.
\end{equation}
We note that each partition of length at most $n$ can be written as a unique linear combination with positive integer coefficients of the integer vectors $\omega_r = e_1 +\cdots + e_r$ with $r = 1,\ldots,n$. To prove the statement it is thus sufficient to verify that both sides of \eqref{TildeNormFactorQuotient} transform in the same manner under the replacement of $\mu$ by $\mu + \omega_r$ for each $r = 1,\ldots,n$. Using Theorems \ref{RecurRelsThm} and \ref{orthogonalityThm} we deduce that
\begin{equation*}
\begin{split}
	\tilde{N}_{\mu+\omega_r} &= \frac{\langle \hat{E}^B_r Y_\mu, Y_{\mu+\omega_r}\rangle_{a,\kappa}}{\hat{V}^B_{\lbrace 1,\ldots,r\rbrace,\lbrace r+1,\ldots,n\rbrace}(\rho^B+\mu)}\\ &= \frac{\hat{V}^B_{\lbrace 1,\ldots,r\rbrace,\lbrace r+1,\ldots,n\rbrace}(-\rho^B-\mu-\omega_r)}{\hat{V}^B_{\lbrace 1,\ldots,r\rbrace,\lbrace r+1,\ldots,n\rbrace}(\rho^B+\mu)}\tilde{N}_\mu,
\end{split}
\end{equation*}
where we used the expansion of $\hat{E}^B_r Y_\mu$ to obtain the first equality, and that of $\hat{E}^B_r Y_{\mu+\omega_r}$ to obtain the second. It follows from Lemma \ref{differenceEqsLemma} that
\begin{equation*}
	\hat{V}^B_{\lbrace 1,\ldots,r\rbrace,\lbrace r+1,\ldots,n\rbrace}(z) = \frac{\hat{\Delta}^B_+(z+\omega_r)}{\hat{\Delta}^B_+(z)},
\end{equation*}
and that
\begin{equation*}
	\hat{V}^B_{\lbrace 1,\ldots,r\rbrace,\lbrace r+1,\ldots,n\rbrace}(-z-\omega_r) = (-1)^r\frac{\hat{\Delta}^B_-(z+\omega_r)}{\hat{\Delta}^B_-(z)}.
\end{equation*}
We thus conclude that
\begin{equation*}
	(-1)^{|\omega_r|}\frac{\hat{\Delta}^B_-(\rho^B+\mu+\omega_r)}{\hat{\Delta}^B_+(\rho^B+\mu+\omega_r)} = \frac{\hat{V}^B_{\lbrace 1,\ldots,r\rbrace,\lbrace r+1,\ldots,n\rbrace}(-\rho^B-\mu-\omega_r)}{\hat{V}^B_{\lbrace 1,\ldots,r\rbrace,\lbrace r+1,\ldots,n\rbrace}(\rho^B+\mu)} \frac{\hat{\Delta}^B_-(\rho^B+\mu)}{\hat{\Delta}^B_+(\rho^B+\mu)},
\end{equation*}
and the statement follows.
\end{proof}

There remains only to compute the normalisation factor $N_{(0)}$.

\begin{lemma}\label{ConstNormFactorLemma}
Assume that $\kappa\geq 0$ and that $a < -2\kappa(n-1) + 1$. Then
\begin{equation}\label{ConstNormFactor}
	N_{(0)} = 2^{\kappa n(n-1)+(a-1)n}n!\prod_{i=1}^n\frac{\Gamma(i\kappa)\Gamma(-a-\kappa(n+i-2)+1)}{\Gamma(\kappa)}.
\end{equation}
\end{lemma}

\begin{proof}
We first observe that Lemma \ref{L2Lemma} ensures that $N_{(0)}$ is well-defined. Changing variables to $(y_1,\ldots,y_n) = (2/x_1,\ldots,2/x_n)$ we find that
\begin{equation*}
	N_{(0)} = 2^{\kappa n(n-1)+(a-1)n}\int_0^\infty dy_1\cdots \int_0^\infty dy_n \prod_{i=1}^n y_i^{-a-2\kappa(n-1)}e^{-y_i}\prod_{i<j}|y_i - y_j|^{2\kappa}.
\end{equation*}
As indicated above, this integral is a well-known limiting case of the Selberg integral. Indeed, by specialising to the root system $B_n$ in Conjecture 6.1 in Macdonald \cite{Mac82} it is readily inferred that
\begin{multline*}
	\int_0^\infty dy_1\cdots \int_0^\infty dy_n \prod_{i=1}^n y_i^{\alpha-\kappa(n-1)-1} e^{-y_i}\prod_{i<j}|y_i - y_j|^{2\kappa}\\ = n!\prod_{i=1}^n\frac{\Gamma(i\kappa)\Gamma(\alpha-\kappa(i-1))}{\Gamma(\kappa)}
\end{multline*}
from which the statement clearly follows.
\end{proof}

We proceed to combine Lemmas \ref{NormFactorQuotientLemma} and \ref{ConstNormFactorLemma} to deduce a fully explicit expression for the squared norms $N_\lambda$.

\begin{proposition}\label{normFactorsProp}
Assume that $\kappa\geq 0$ and that $a$ satisfies the condition \eqref{aCond}. Let $\lambda = (\lambda_1,\ldots,\lambda_n)$ be a partition such that $|\lambda|\leq m$. Then
\begin{multline*}
	N_\lambda = c_\lambda\prod_{1\leq i<j\leq n}\Bigg(\frac{\Gamma(\kappa(j-i+1)+\lambda_i-\lambda_j)\Gamma(\kappa(j-i-1)+1+\lambda_i-\lambda_j)}{\Gamma(\kappa(j-i)+\lambda_i-\lambda_j)\Gamma(\kappa(j-i)+1+\lambda_i-\lambda_j)}\\ \qquad\qquad\qquad\quad \times\frac{\Gamma(-a-\kappa(2n-i-j)+2-\lambda_i-\lambda_j)}{\Gamma(-a-\kappa(2n-i-j+1)+2-\lambda_i-\lambda_j)}\\ \shoveright{\times\frac{\Gamma(-a-\kappa(2n-i-j)+1-\lambda_i-\lambda_j)}{\Gamma(-a-\kappa(2n-i-j-1)+1-\lambda_i-\lambda_j)}\Bigg)}\\ \times\prod_{i=1}^n\Bigg(\frac{\Gamma(-a-2\kappa(n-i)+2-2\lambda_i)\Gamma(-a-2\kappa(n-i)+1-2\lambda_i)}{\Gamma(-a-\kappa(n-i)+2-\lambda_i)}\\ \times\Gamma(\kappa(n-i)+1+\lambda_i)\Bigg)
\end{multline*}
with
\begin{equation*}
	c_\lambda = 2^{2|\lambda|+\kappa n(n-1)+(a-1)n} n!.
\end{equation*}
In particular, it is manifest that $N_\lambda(a,\kappa)$ is a positive analytic function of $a$ and $\kappa$ (under the assumed conditions on these parameters).
\end{proposition}

\begin{proof}
We will use the notation $(\alpha)_m$ for the falling factorial, defined by $(\alpha)_0 = 1$ and
\begin{equation*}
	(\alpha)_m = \alpha(\alpha - 1)\cdots (\alpha -m + 1) = \frac{\Gamma(\alpha+1)}{\Gamma(\alpha-m+1)}
\end{equation*}
for $m>0$. In addition, we let $(\alpha_1,\ldots,\alpha_p)_m = (\alpha_1)_m\cdots(\alpha_p)_m$, and similarly for the Pochhammer symbol $\lbrack\alpha\rbrack_m$. It follows from the formulae $\Gamma(\alpha+m)/\Gamma(\alpha) = \lbrack\alpha\rbrack_m$ and $(-1)^m\lbrack\alpha\rbrack_m = (-\alpha)_m$ that
\begin{multline*}
	(-1)^{|\lambda|}\frac{\hat{\Delta}^B_+(\rho^B+\lambda)\hat{\Delta}^B_-(\rho^B+\lambda)}{\hat{\Delta}^B_+(\rho^B)\hat{\Delta}^B_-(\rho^B)} = \prod_{i<j}\frac{\big\lbrack\kappa(j-i+1),\kappa(j-i-1)+1\big\rbrack_{\lambda_i-\lambda_j}}{\big\lbrack\kappa(j-i),\kappa(j-i)+1\big\rbrack_{\lambda_i-\lambda_j}}\\ \times \frac{\big(-a-\kappa(2n-i-j+1)+1,-a-\kappa(2n-i-j-1)\big)_{\lambda_i+\lambda_j}}{\big(-a-\kappa(2n-i-j)+1,-a-\kappa(2n-i-j)\big)_{\lambda_i+\lambda_j}}\\ \times \prod_{i=1}^n \frac{\big\lbrack\kappa(n-i)+1\big\rbrack_{\lambda_i}\big(-a-\kappa(n-i)+1\big)_{\lambda_i}}{\big(-a-2\kappa(n-i)+1,-a-2\kappa(n-i)\big)_{2\lambda_i}}.
\end{multline*}
Using the identity
\begin{equation*}
	\prod_{1\leq i<j\leq n}\frac{\Gamma(\alpha(j-i+1)+\beta)}{\Gamma(\alpha(j-i)+\beta)} = \prod_{i=1}^n\frac{\Gamma(\pm i\alpha+\beta)}{\Gamma((1-i\pm i)\alpha+\beta)},
\end{equation*}
valid for generic $\alpha$ and $\beta$, it is readily verified that
\begin{equation*}
	\prod_{1\leq i<j\leq n}\frac{\Gamma(\kappa(j-i))\Gamma(\kappa(j-i)+1)}{\Gamma(\kappa(j-i+1))\Gamma(\kappa(j-i-1)+1)} = \prod_{i=1}^n\frac{\Gamma(\kappa)\Gamma(\kappa(i-1)+1)}{\Gamma(i\kappa)},
\end{equation*}
and that
\begin{multline*}
	\prod_{1\leq i<j\leq n}\frac{\Gamma(-a-\kappa(2n-i-j+1)+2)\Gamma(-a-\kappa(2n-i-j-1)+1)}{\Gamma(-a-\kappa(2n-i-j)+2)\Gamma(-a-\kappa(2n-i-j)+1)} =\\ \prod_{i=1}^n\frac{\Gamma(-a-\kappa(2n-i)+2)\Gamma(-a-\kappa(2n-2i)+1)}{\Gamma(-a-\kappa(2n-2i+1)+2)\Gamma(-a-\kappa(2n-i-1)+1)}.
\end{multline*}
By cancelling common factors in the resulting nominator and denominator, and using the difference equation $\Gamma(z+1) = z\Gamma(z)$, it is now a straightforward exercise to verify that Lemmas \ref{NormFactorQuotientLemma} and \ref{ConstNormFactorLemma} combine to give the statement.
\end{proof}

\section{Concluding remarks}
In this final section we briefly return to the discussion of the relation between the multivariable Bessel polynomials and the hyperbolic Sutherland model with external Morse potential, and present some remarks on the problem of constructing a (moment) functional with respect to which all multivariable Bessel polynomials would be orthogonal.

\subsection{Eigenfunctions of the hyperbolic Sutherland model}
We will throughout this section assume that $\kappa > 3/2$. To each permutation $P\in S_n$ we associate the following open subset of $\mathbb{R}^n$:
\begin{equation*}
	\Delta_P = \lbrace z\in\mathbb{R}^n: z_{P(1)} <\cdots < z_{P(n)}\rbrace.
\end{equation*}
In addition, we assume that the parameters $a$ and $\kappa$ are such that there exist at least one partition $\lambda = (\lambda_1,\ldots,\lambda_n)$ which satisfies the inequality
\begin{equation}\label{parameterCond}
	a < -2(|\lambda| + \kappa(n-1)) + 1.
\end{equation}
For each such partition $\lambda$ and permutation $P\in S_n$ we have the following eigenfunction of the Schr\"odinger type operator \eqref{SchrodOp}:\begin{equation}\label{HEigenfunc}
	\Psi^{(P)}_\lambda(z) = \left\lbrace\begin{array}{ll}
		\frac{n!}{N_\lambda}\Psi_0(e^z)Y_\lambda(e^z), & z\in\Delta_P,\\
		0, & z\notin\Delta_P,
	\end{array}\right.
\end{equation}
with $N_\lambda$ as in Proposition \ref{normFactorsProp}. It is clear from the discussion in Section 7 that this eigenfunction is normalised to one. We remark that all the second order derivatives of the eigenfunction \eqref{HEigenfunc} are contained in the Hilbert space $L^2(\mathbb{R}^n,dz_1\cdots dz_n)$ precisely for $\kappa > 3/2$. In addition, it seems that for $\kappa<3/2$, or rather for $2\kappa(\kappa - 1)<3/2$, it is a rather delicate matter to associate a domain to \eqref{SchrodOp} such that the resulting Schr\"odinger operator is self-adjoint; c.f.~Feh\'er et al.~\cite{FTF05} which contains a detailed investigation of the equivalent problem in the case of the three particle Calogero model.

The eigenfunction \eqref{HEigenfunc} of the Schr\"odinger type operator \eqref{SchrodOp} has the following natural physical interpretation: it represents a bound state in the corresponding quantum many-body system where all particles are confined to the region $\Delta_P$. Indeed, the fact that $\Psi_0$ vanish at each hyperplane $z_i = z_j$ with $i,j = 1,\ldots,n$ such that $i\neq j$ implies that the probability flux across such a hyperplane is zero; see e.g.~Section 2.4 in Sakurai \cite{Sak94} for a definition of the probability flux. In more concrete terms this means that the particles can not pass each other. Consequently, once confined in a given region $\Delta_P$ they will remain there indefinitely. For a further discussion of this point see e.g.~Calogero \cite{Cal71}.

We note that for given values of the parameters $a$ and $\kappa$ there are only a finite number of partitions $\lambda$ that satisfy \eqref{parameterCond}, i.e., the Schr\"odinger type operator \eqref{SchrodOp} has only a finite number of eigenfunctions of the form \eqref{HEigenfunc}. This is a manifestation of the fact that the Morse potential only supports finitely many bound states.

\subsection{The multivariable Bessel moment problem}
As we mentioned at the beginning of Section 7, in the one-variable case it is well known that there exist no (positive) measure on the real line with respect to which all Bessel polynomials are orthogonal. This naturally leads to the question of whether they form an orthogonal system in some more general sense. To be more precise, one might inquire whether or not there exists a non-trivial (moment) functional
\begin{equation*}
	I: \Lambda_n\rightarrow \mathbb{C}
\end{equation*}
with respect to which all multivariable Bessel polynomials are orthogonal. For one variable this problem is known to have a positive solution, and a number of explicit integral representations of such a functional have been constructed. We mention, in particular, the work of Krall and Frink \cite{KF49}, who obtained an integral representation with integration along the unit circle, and Evans et al. \cite{EEKKL92}, in which the integration is along the positive real line. In this section we will briefly consider this problem in the two-variable case.

We thus set $n = 2$, and recall that the algebra $\Lambda_2$ is freely generated by the elementary symmetric functions
\begin{equation*}
	e_1 = x_1 + x_2,\quad e_2 = x_1x_2;
\end{equation*}
see e.g.~Section I.2 in Macdonald \cite{Mac95}. It is readily verified that, as a linear operator on $\Lambda_2$, $D^B$ is given by
\begin{multline*}
	D^B = (e_1^2 - 2e_2)\partial_{e_1}^2 + 2e_1e_2\partial_{e_1}\partial_{e_2} + 2e_2^2\partial_{e_2}^2\\ + (1 + (2\kappa + a)e_1)\partial_{e_1} + (e_1 + 2(\kappa + a)e_2)\partial_{e_2},
\end{multline*}
where we have used the notation $\partial_{e_1} = \partial/\partial e_1$, and similarly for $\partial_{e_2}$. For a functional $I: \Lambda_2\rightarrow \mathbb{C}$ and $p\in\Lambda_2$ we define $\partial_{e_1} I$, $\partial_{e_2} I$ and $pI$ by setting
\begin{equation*}
	\partial_{e_1}I(q) = -I(\partial_{e_1}q),\quad \partial_{e_2}I(q) = -I(\partial_{e_2}q),\quad pI(q) = I(pq),
\end{equation*}
for all $q\in\Lambda_2$. With this definition in place, it is a matter of a straightforward computation to verify that $DB$ is symmetric with respect to $I$, i.e., $I(qD^Bp) = I(pD^Bq)$ for all $p,q\in\Lambda_2$, if and only if $I$ satisfies the differential equations
\begin{align*}
	\partial_{e_1}(e_1^2 - 2e_2)I + \partial_{e_2}e_1e_2I - (1 + (2\kappa + a)e_1)I &= 0,\\
	\partial_{e_1}e_1e_2I + 2\partial_{e_2}e_2^2I - (e_1 + 2(\kappa + a)e_2)I &= 0.
\end{align*}
We let $I_{nm} = I(e_1^ne_2^m)$ be the moments of $I$ with respect to the linear basis of $\Lambda_2$ formed by the symmetric polynomials $e_1^ne_2^m$ with $n,m\in\mathbb{N}$. The differential equations for $I$ stated above then translate into the following two recurrence relations for the moments $I_{nm}$:
\begin{subequations}\label{momentRecurRels}
\begin{align}
	(n + m - (2\kappa + a))I_{n+1,m} - 2nI_{n-1,m+1} - I_{nm} &= 0,\\
	(n + 2m - 2(\kappa + a))I_{n,m+1} - I_{n+1,m} &= 0.
\end{align}
\end{subequations}
By induction in $n + m$ it is a simple exercise to show that once the moment $I_{0,0}$ is fixed these recurrence relations uniquely determine the remaining moments $I_{nm}$. It is easily seen that for generic values of the parameters $a$ and $\kappa$  (and $n = 2$) the eigenvalues of $D^B$ are non-degenerate. Assuming this to be the case, the two-variable Bessel polynomials thus form an orthogonal system with respect to a functional $I$ determined by the recurrence relations \eqref{momentRecurRels} and the moment $I_{0,0}$.

From the discussion above it is not clear whether such a functional $I$ might have a natural integral representation, or even if its moments (with respect to the elementary symmetric functions) have a simple expression. It is furthermore not at all clear what happens for $n$ greater than two. We plan to return to these questions in a forthcoming paper \cite{Hal}.

\section*{Acknowledgements.}
I would like to thank J.~F.~van Diejen for helpful and inspiring discussions, in particular, on topics relating to the material in Section 7. I would also like to thank the Institute of Mathematics at the University of Talca, where parts of this paper was written, for its hospitality. Financial support from the European Union through the FP6 Marie Curie RTN ENIGMA (Contract number MRTN-CT-200405652) is gratefully acknowledged.

\bibliographystyle{amsalpha}

\end{document}